\documentclass[11pt,a4paper]{amsart}
\usepackage{amssymb,amsmath,amsthm}

\usepackage[pdftex]{graphicx}

\usepackage[colorlinks=true, pdfstartview=FitV, linkcolor=blue, citecolor=blue, urlcolor=blue,pagebackref=false]{hyperref}

\usepackage{esint}
\usepackage{mathrsfs}

\usepackage{mathtools}

\usepackage{times,latexsym,color}

\newcommand{\avg}{{\rm avg}}
\newcommand{\BOX}{\ensuremath\Box}

\newtheorem{theorem}{Theorem}
\newtheorem*{theorem*}{Theorem}
\newtheorem{pro}{Proposition}
\newtheorem{lemma}{Lemma}
\newtheorem{cor}{Corollary}

\theoremstyle{remark}
\newtheorem{remark}{Remark}

\theoremstyle{definition}

\DeclareMathOperator{\dist}{dist}

\DeclareMathOperator{\supp}{supp}

\def\XXint#1#2#3{{\setbox0=\hbox{$#1{#2#3}{\int}$ }
\vcenter{\hbox{$#2#3$ }}\kern-.6\wd0}}

\newcommand{\N}{\mathbb{N}}

\newcommand{\R}{\mathbb{R}}

\newcommand{\ep}{\varepsilon}
\newcommand{\Div}{\textrm{div}}
\newcommand{\p}{\partial}

\newcommand{\norm}[1]{\lVert #1 \rVert}

\definecolor{darkgreen}{rgb}{0,0.5,0}
\definecolor{darkblue}{rgb}{0,0,0.7}
\definecolor{darkred}{rgb}{0.9,0.1,0.1}
\definecolor{lightblue}{rgb}{0,0.51,1}

\DeclarePairedDelimiter\ceil{\lceil}{\rceil}
\DeclarePairedDelimiter\floor{\lfloor}{\rfloor}

\begin{document}

\title{{Localized Quantitative Estimates and potential blow-up rates for the Navier-Stokes equations }}

\author[T. Barker]{Tobias Barker\\\\\textbf{\,\,\,\,\,\,\,\,\,\,\,\,\,\,\,\,\,\,\,\,\,\,\,\,\,\,\,\,\,\,\,\,\,\,\,\,\,\,\,\,\,\,\,\,\,\,\,In memory of my Grandfather Maurice (1930-2022)}}
\address[T. Barker]{Department of Mathematical Sciences, University of Bath, Bath BA2 7AY. UK}
\email{tobiasbarker5@gmail.com}

\subjclass[2010]{}
\date{\today}

\maketitle
\begin{abstract}
We show that if $v$ is a smooth suitable weak solution to the Navier-Stokes equations on $B(0,4)\times (0,T_*)$, that possesses a singular point $(x_0,T_*)\in B(0,4)\times \{T_*\}$, then for all $\delta>0$ sufficiently small one necessarily has
$$\lim\sup_{t\uparrow T_*} \frac{\|v(\cdot,t)\|_{L^{3}(B(x_0,\delta))}}{\Big(\log\log\log\Big(\frac{1}{(T_*-t)^{\frac{1}{4}}}\Big)\Big)^{\frac{1}{1129}}}=\infty.$$
This \textit{local} result improves upon the corresponding \textit{global} result recently established by Tao \cite{tao}. The proof is based upon a quantification of Escauriaza, Seregin and \v{S}verak's \textit{qualitative} local result \cite{ESS}. In order to prove the required localized quantitative estimates, we show that in certain settings one can quantify a \textit{qualitative} truncation/localization procedure introduced by Neustupa and Penel \cite{neustupapenel}. After performing the quantitative truncation procedure, the remainder of the proof hinges on a physical space analogue of Tao's breakthrough strategy, established by Prange and the author \cite{barkerprange2020}.

\end{abstract}
\section{Introduction}
In \cite{Le}, Leray showed that for any square integrable weakly divergence-free initial data, there exists an associated finite-energy solution\footnote{For a definition of finite energy solutions see subsection \ref{finiteneergyswsnotation}.} $v:\mathbb{R}^3\times (0,\infty)\rightarrow\mathbb{R}^3 $ of the three-dimensional Navier-Stokes equations
\begin{equation}\label{NSEintro}
\partial_{t}v-\Delta v+v\cdot\nabla v+\nabla p=0,\quad\textrm{div}\,v=0,\quad v(x,0)=v_0(x)\quad\textrm{in}\quad \mathbb{R}^3\times (0,\infty).
\end{equation}
The above equations are invariant under the \textit{Navier-Stokes rescaling}
\begin{equation}\label{NSrescale}
(v_{\lambda}(x,t),p_{\lambda}(x,s), v_{0\lambda}(x)):=(\lambda v(\lambda x,\lambda^2 t), \lambda^2p(\lambda x,\lambda^2 t), \lambda v_{0}(\lambda x))\quad\textrm{with}\,\,\lambda>0.
\end{equation}
After almost 90 years, it remains unknown if finite energy solutions of \eqref{NSEintro}, with sufficiently smooth initial data, remain smooth for all times.

This paper concerns \textit{localized quantitative behavior} for potential blow-up solutions of \eqref{NSEintro} in terms of the $L^3$ norm, which is invariant with respect to the rescaling for the initial data \eqref{NSrescale} and is referred to as a \textit{critical norm}.

The first quantitative blow-up rates for the Navier-Stokes equations were established by Leray. Leray showed that if a finite-energy solution of the Navier-Stokes equations first loses smoothness at $T_*>0$ then for $3<p\leq\infty$
\begin{equation}\label{Leraylowerbound}
\|v(\cdot,t)\|_{L^{p}(\mathbb{R}^3)}\geq \frac{C(p)}{(T_* -t)^{\frac{1}{2}(1-\frac{3}{p})}}\quad\textrm{for}\,\,\textrm{every}\,\,0\leq t<T_*.
\end{equation}
This can be deduced from the fact that for $3<p\leq\infty$, a finite-energy solution of \eqref{NSEintro} (with sufficiently smooth initial data) stays smooth on $(0,T)$ with 
\begin{equation}\label{existsubcritlower}
T\geq C_p\|v_0\|_{L^{p}(\mathbb{R}^3)}^{-\frac{2}{1-\frac{3}{p}}}.
\end{equation}
The behavior of the critical $L^{3}$ norm of $v$ is more subtle and the above argument cannot be used to deduce that it necessarily diverges near a potential blow-up time. As observed in \cite{merleraphael}, this obstruction also occurs for critical norms of other evolution equations with a scaling symmetry (see also \cite{barkerox}).

The first \textit{qualitative} blow-up results for the critical $L^{3}$ norm of the Navier-Stokes equations was shown in a celebrated paper by Escauriaza, Seregin and \v{S}ver\'{a}k in \cite{ESS}. In \cite{ESS}, a contradiction argument using backward uniqueness and unique continuation for parabolic operators was used to show that if a finite-energy solution on $(0,\infty)$ first loses smoothness at $T_*$ then for any singular point\footnote{For the definition of a singular point, see subsection \ref{singpointsnotation}.} $(x_0,T_*)$ and any $\delta>0$ one necessarily has
\begin{equation}\label{ESSresult}
\limsup_{t\uparrow T_*} \|v(\cdot,t)\|_{L^{3}(B(x_0,\delta))}=\infty.
\end{equation}
This follows from  a regularity criteria established in the same paper. Namely, for a suitable weak solution\footnote{For a definition of suitable weak solutions see subsection \ref{finiteneergyswsnotation}.} on $B(0,1)\times (-1,0)$ we have the qualitative regularity criteria that
\begin{equation}\label{swsregESS}
v\in L^{\infty}(-1,0; L^{3}(B(0,1)))\Rightarrow v\in L^{\infty}(B(0,r)\times (-r^2,0))\,\,\textrm{for}\,\,\textrm{all}\,\,\textrm{sufficiently}\,\,\textrm{small}\,\,r>0.
\end{equation}
Since \cite{ESS}, there have been many extensions regarding qualitative potential blow-up behavior of critical norms for the Navier-Stokes equations. See \cite{sereginL3limit}, \cite{phuc}. \cite{GKP}, \cite{wangzhang}, \cite{albritton} and \cite{albrittonbarker}. After Escauriaza, Seregin and \v{S}ver\'{a}k's result, the subtle question of the divergence of critical norms near a potential blow-up was examined for other supercritical partial differential equations. See \cite{merleraphael}, \cite{mizoguchi}, \cite{miurasemilinear} and \cite{wangharmonic}, for example.

With the exception of \cite{merleraphael}, the aforementioned results regarding critical norms are based on contradiction arguments and are hence purely qualitative. To the best of the author's knowledge, the first quantitative blow-up rates for critical norms of a supercritical partial differential equation was obtained by Merle and Raphaël in \cite{merleraphael}. In \cite{merleraphael}, a logarithmic blow-up rate was obtained for  the critical norms of radial solutions of the $L^2$ supercritical nonlinear Schr\"{o}dinger equation. In the breakthrough paper \cite{tao}, Tao established the first quantitative blow-up rate for the critical $L^{3}$ norm for finite energy solutions of the Navier-Stokes equations \eqref{NSEintro}. In particular, Tao showed that if  a finite-energy solution on $(0,\infty)$ first loses smoothness at $T_*$ then one necessarily has
\begin{equation}\label{taoblowupintro}
\limsup_{t\uparrow T_*}\frac{\|v(\cdot,t)\|_{L^{3}(\mathbb{R}^3)}}{(\log\log\log(\frac{1}{T_*-t}))^c}=\infty.
\end{equation}
Here, $c>0$ is a universal constant. Tao's result is a consequence of the following quantitative estimate. Namely, for a smooth finite-energy solution to the Navier-Stokes equations on $\mathbb{R}^3\times (-1,0)$ we have that
\begin{equation}\label{taoquantestintro}
\|v\|_{L^{\infty}(-1,0; L^{3}(\mathbb{R}^3))}\leq M\Rightarrow \|v\|_{L^{\infty}(-\tfrac{1}{2},0; L^{\infty}(\mathbb{R}^3))}\leq \exp\exp\exp(M^{C}),
\end{equation}
where $C$ is a positive universal constant.
The  above quantitative blow-up rate and estimate have been subsequently extended and generalized in \cite{barkerprange2020}, \cite{palasekARMA}-\cite{palasekNSEd} and \cite{he}.

In light of Escauriaza, Seregin and \v{S}ver\'{a}k's blow up criteria \eqref{ESSresult} and regularity criteria \eqref{swsregESS}, it is natural to ask the following questions.
\begin{enumerate}
\item[] 1) Can Escauriaza, Seregin and Sver\'{a}k's regularity criteria \eqref{swsregESS} be quantified?
\medskip
\item[] 2) Can we obtain quantiative blow-up rates for the critical $L^{3}$ norm in any local neighborhood of a singular point?
\end{enumerate}
Question 2) is also of potential interest in the numerical search for potentially singular solutions  of the Navier-Stokes equations \eqref{NSEintro}. In \cite{hou}, the growth of the localized $L^3$ norm is investigated for certain numerical solutions for the three-dimensional axisymmetric solutions of the Navier-Stokes equation. In this paper we provide an answer to questions 1)-2) by means of the following results.

\subsection{Main Results}
Our first result quantifies Escauriaza, Seregin and \v{S}ver\'{a}k's regularity criteria \eqref{swsregESS}, thus providing a positive answer to question 1) in the previous subsection.
\begin{theorem}\label{locest}
Suppose that $(v,p)$ is a smooth solution to the Navier-Stokes equations on $B(0,4)\times (-16,0]$ and that $(v,p)$ is a suitable weak solution on $Q(0,4):=B(0,4)\times (-16,0)$.\\
Furthermore, suppose that $\mathcal{M}\in (0,\infty)$ is such that
\begin{itemize}
\item $\|v\|_{L^{\infty}_{t}L^{3}_{x}(Q(0,4))}\leq\mathcal{M}$ and
\item $$\max\{\mathcal{M}_{0}, \|p\|_{L^{\frac{3}{2}}_{x,t}(Q(0,4))}^{\frac{3}{2}}\}\leq \mathcal{M}. $$
\end{itemize}
Here, $\mathcal{M}_{0}$ is a sufficiently large universal constant.

Then  we conclude that
\begin{equation}\label{locesteqn}
\|v\|_{L^{\infty}_{x,t}(B(0,\frac{1}{2})\times (-e^{-e^{e^{\mathcal{M}^{1128}}}},0))}\leq e^{e^{e^{\mathcal{M}^{1128}}}}.
\end{equation}
\end{theorem} 
As a consequence of the quantitative estimate \eqref{locesteqn}, we obtain quantitative blow-up rates of the critical $L^{3}$ norm in any local neighborhood of a singular point. This improves Tao's necessary condition \eqref{taoblowupintro} and quantifies Escauriaza, Seregin and \v{S}ver\'{a}k's qualitative necessary condition \eqref{ESSresult}.
\begin{theorem}\label{locaquantrate}
Suppose that $(v,p)$ is a smooth suitable weak solution to the Navier-Stokes equations on $B(0,4)\times (0,T_*)$.\\
Furthermore suppose that
\begin{itemize}
\item $v\in L^{\infty}_{t,loc}([0,T_*); L^{\infty}(B(0,4)))$ and
\item there exists $x_0\in B(0,4)$ such that for all $r\in (0,\dist(x_0, \partial B(0,4))$ $$v\notin L^{\infty}_{x,t}(B(x_0,r)\times (T_*-r^2,T_*)).$$ 

\end{itemize} 
Then, under the above assumptions we conclude that for all $\delta\in (0,\dist(x_0, \partial B(0,4))$ we have
\begin{equation}\label{locquantrateeqn}
\lim\sup_{t\uparrow T_*} \frac{\|v(\cdot,t)\|_{L^{3}(B(x_0,\delta))}}{\Big(\log\log\log\Big(\frac{1}{(T_*-t)^{\frac{1}{4}}}\Big)\Big)^{\frac{1}{1129}}}=\infty.
\end{equation} 
\end{theorem}
\subsection{Comparison to Previous Literature and Novelty of Our Approach} 
As previously mentioned, Tao's global quantitative potential blow-up rate \eqref{taoblowupintro} relies upon the quantitative estimate \eqref{taoquantestintro}, which is achieved via a strategy\footnote{For a detailed description of Tao's strategy, we refer the reader to \cite{tao} and \cite{barkerprange2020}.} using quantitative Carleman inequalities for parabolic differential inequalities. A key part of Tao's strategy is using that a smooth solution $v:\mathbb{R}^3\times [0,1]\rightarrow\mathbb{R}^3$ with \textit{global bound}
\begin{equation}\label{criticalboundintrotao}
\|v\|_{L^{\infty}_{t}L^{3}_{x}(\mathbb{R}^3\times (0,1))}\leq M 
\end{equation}
possesses space-time regions of regularity, which can be quantified explicitly in terms of $M$.

An elementary approach to obtaining  quantitative regions of space-time regularity in this setting was provided in the subsequent paper \cite{barkerprange2020}. In \cite{barkerprange2020}, it was used that the global functions $(v,p)$ satisfying \eqref{criticalboundintrotao} also satisfy
\begin{equation}\label{expmanyscales}
\int\limits_{\frac{1}{2}}^{1} \int\limits_{B(0,\exp(M^{C}))} |v|^{\frac{10}{3}}+|p|^{\frac{5}{3}} dxdt\lesssim M^{O(1)}.
\end{equation}
Here, $C$ is a positive universal constant.
In particular, the estimate \eqref{expmanyscales} occurs for an integral over a region with exponentially more spatial scales compared to the right hand side bound. This then allows the use of the pigeonhole principle and Caffarelli, Kohn and Nirenberg's $\varepsilon$-regularity criteria \cite{CKN} to obtain that $v$ is quantitatively bounded in quantitative space-time annuli. 

In the context of a smooth suitable weak solution $(v,p)$ on $B(0,4)\times (-16,0]$ with
$$\|v\|_{L^{\infty}_{t}L^{3}_{x}(B(0,4)\times (-16,0))}\leq \mathcal{M} $$ ($\mathcal{M}$ is as in Theorem \ref{locest}), then $(v,p)$ does not satisfy \eqref{expmanyscales} but instead obeys the Morrey bound
\begin{equation}\label{103Morrey}
\sup_{(x_0,t_0)\in B(0,1)\times (-1,0], 0<r\leq 1}\frac{1}{r^{\frac{5}{3}}}\int\limits_{-r^2}^{0}\int\limits_{B(0,r)} |v|^{\frac{10}{3}}+|p|^{\frac{5}{3}} dxdt\lesssim \mathcal{M}^{O(1)}.
\end{equation}
In contrast to the global setting \eqref{expmanyscales}, \eqref{103Morrey} is integrated over a region that is related to the right hand side bound by a power law. This prevents us applying the same arguments as in the global case (\cite{tao} and \cite{barkerprange2020}) to obtain quantitative bounds for $v$ on quantitative space-time annuli. This represents a major block when attempting to apply the previous strategies (\cite{tao} and \cite{barkerprange2020}) for producing quantitative estimates to the local setting of suitable weak solutions.

A systematic bridge between global regularity results and local regularity results was established by Neustupa and Penel in \cite{neustupapenel}, in the context of one-component regularity criteria for the Navier-Stokes equations, which we will refer to as a `truncation procedure'. This bridge relies upon the following fact for suitable weak solutions. Namely, if $(v,p)$ is a suitable weak solution on $B(0,1)\times (-1,0)$ then there exists $0<r_1<r_2<1$ and $0<\delta_1<1$ such that
\begin{equation}\label{vboundedCKN}
v\in L^{\infty}_{x,t}(B(0,r_2)\setminus \overline{B(0,r_1)}\times (-\delta_1,0)).
\end{equation}
This relies upon the celebrated fact (proven in \cite{CKN}) that the parabolic one-dimensional Hausdorff dimension of the singular set is zero. Therefore, \eqref{vboundedCKN} is purely qualitative.

Having established \eqref{vboundedCKN}, in \cite{neustupapenel} an appropriate cut-off function $\Phi$ is chosen and the Bogovskiĭ operator \cite{bogovskii} is utilized to `truncate' the original suitable weak solution $v$ to obtain a global function $u:\mathbb{R}^3\times (-\delta_1,0)\rightarrow\mathbb{R}^3$ possessing the following properties.
\begin{itemize}
\item $u$ agrees with $v$ on $B(0,r_1)\times (-\delta_1,0)$.
\medskip
\item $u$ solves the Navier-Stokes equations with compactly supported forcing $f:\mathbb{R}^3\times (-\delta_1,0)\rightarrow\mathbb{R}^3$.
\medskip
\item $f$ belongs to subcritical spaces, with respect to the Navier-Stokes rescaling for the forcing\footnote{For the notion of subcritical forcing, see subsection \ref{singpointsnotation}.}. 
\end{itemize} 
The fact that $f$ belongs to subcritical spaces means that it does not infer with the regularity properties of $u$ and one can then  obtain apriori estimates for the global function $u$. The main difficulty with applying Neustupa and Penel's above truncation procedure to obtain the localized quantitative estimates in Theorem \ref{locest} is as follows. Neustupa and Penel's truncation procedure cannot be quantified in the general case, since it relies upon the qualitative property \eqref{vboundedCKN} coming from Caffarelli, Kohn and Nirenberg's qualitative result \cite{CKN}.

To overcome the aforementioned difficulties in producing the local quantitative estimates in Theorem \ref{locest}, we show that under the assumptions of Theorem \ref{locest} Neustupa and Penel's truncation procedure can be fully quantified (Proposition \ref{truncationprocedure}). A key part of quantifying the above truncation procedure is the application of local-in-space smoothing near the initial time for the Navier-Stokes equations, which was initiated in Jia and \v{S}ver\'{a}k's seminal work \cite{jiasverak} and subsequently extended in \cite{barkerprangeARMA}, \cite{KMT21IMRN}-\cite{KMT21} and \cite{albrittonbarkerprange}.

Having quantified Neustupa and Penel's truncation procedure, we get a solution to the Navier-Stokes equations $V:\mathbb{R}^3\times (-2,0)\rightarrow \mathbb{R}^3$ with forcing $F:\mathbb{R}^3\times (-2,0)\rightarrow\mathbb{R}^3$ such that $V$ has quantitatively controlled norms and $F$ has quantitatively controlled (negligible) subcritical norms. This allows us to apply energy methods to $V$ and then using $\varepsilon$-regularity in a similar way to \cite{barkerprange2020}, we can obtain quantitative space-time regions of space-time regularity for $V$. This then allows us to apply the strategy for producing quantitative estimates in \cite{barkerprange2020}\footnote{This is a physical space analogue of Tao's strategy in \cite{tao}.} to $V$ instead of $v$, which enables us to obtain Theorem \ref{locest}. In order to apply the strategy in \cite{barkerprange2020}, some bookkeeping is required to ensure that the support in space of the forcing $F$ does not intersect with the space-time regions where the quantitative Carleman inequalities are applied to $\Omega=\nabla\times V$.

\subsection{Extensions and Open Problems}

Tao's quantitative blow-up rate for the $L^{3}$ norm of the Navier-Stokes equations near a blow-up time has been improved in \cite{barkerprange2020}, \cite{palasekARMA}-\cite{palasekNSEd} and \cite{he}. In \cite{barkerprange2020}, it was shown that if $v:\mathbb{R}^3\times(0,T_*)\rightarrow\mathbb{R}^3$ is a smooth solution to the Navier-Stokes equations on $\mathbb{R}^3\times (0,T_*)$, which satisfies the \textit{global} bound
\begin{equation}\label{typeIglobal}
\|v\|_{L^{\infty}(0,T_*; L^{3,\infty}(\mathbb{R}^3))}\leq M
\end{equation}
and possesses a singular point at $(x,t)=(0,T_*)$ then
\begin{equation}\label{typeIblowuprate}
\int\limits_{B(0,R)} |v(x,t)|^3 dx\geq \exp(-\exp(M^{C}))\log\Big(\frac{R^2M^{-C}}{T_*-t}\Big)
\end{equation}
for all $t$ sufficiently close to $T_*$. Here, $C$ is a positive universal constant. In \cite{palasekARMA} it was shown that if $v:\mathbb{R}^3\times(0,\infty)\rightarrow\mathbb{R}^3$ is a smooth finite-energy axisymmetric solution, which first loses smoothness at $T_*$, then
\begin{equation}\label{axiblowuprate}
\limsup_{t\uparrow T_*} \frac{\|v(\cdot,t)\|_{L^{3}(\mathbb{R}^3)}}{\big(\log\log\big(\frac{1}{T_*-t}\big)\big)^c}=\infty.
\end{equation} 
Here, $c$ is a positive universal constant.

Using the methodology of this paper, we expect that these results can be extended by being localized. For the blow-up rate \eqref{typeIblowuprate}, it seems plausible that the global assumption \eqref{typeIglobal} can be replaced by the local assumption 
$$\|v\|_{L^{\infty}(T_*-\delta^2, T_*; L^{3,\infty}(B(0,\delta)))}\leq M\quad\textrm{for}\,\,\delta\in (0,{T}^{\frac{1}{2}}_*). $$
We also expect that for axisymmetric solutions, the blow-up rate \eqref{axiblowuprate} can be refined to obtain
$$ 
\limsup_{t\uparrow T_*} \frac{\|v(\cdot,t)\|_{L^{3}(B(x_0,\delta))}}{\big(\log\log\big(\frac{1}{T_*-t}\big)\big)^c}=\infty\quad\textrm{for}\,\,\delta>0.
$$
Here, $x_0=(0,0,z_0)$ is any potential singular point at time $T_*$. We also expect that similar localized refinements apply to the global results obtained in \cite{palasekNSEd} and \cite{he}.

In \cite{albrittonbarkerlocal}, Neustupa and Penel's \textit{qualitative} truncation procedure was used to show that if $v:\mathbb{R}^3\times (-1,0)\rightarrow\mathbb{R}^3$ is a smooth finite-energy solution on $\mathbb{R}^3\times (-1,0)$ then for any $\delta>0$
\begin{equation}\label{timesliceloc}
\sup_{n}\|v(\cdot,t_{(n)})\|_{L^{3}(B(0,\delta))}=M\quad\textrm{with}\quad t_{(n)}\uparrow 0
\end{equation}
implies that the space-time origin is not a singular point. It remains unclear how to quantify the regularity criteria \eqref{timesliceloc}. In particular, the following remains open.
\begin{itemize}
\item[]\textbf{Open Problem}: Under the assumption \eqref{timesliceloc}, does there exist a function\\ $F:\{(a_{(n)})_{n\in\mathbb{N}}: a_{(n)}\in (-1,0)\}\times [0,\infty)\times [0,\infty)\rightarrow [0,\infty)$ such that
\begin{equation}\label{quanttimeslice}
|v(0,0)|\leq F((t_{(n)})_{n\in\mathbb{N}},M,\delta)?
\end{equation}
\end{itemize}
The difficulty in obtaining \eqref{quanttimeslice} is that under the assumption \eqref{timesliceloc}, $v$ does not necessarily belong to critical Morrey spaces, which is currently needed in this paper to quantify Neustupa and Penel's truncation procedure.
\section{Preliminaries}
\subsection{Notation}
\subsubsection{Universal Constants, Vectors and Domains}\label{univvecnotation}
We use the notation $X\lesssim Y$, which means that there exists a positive universal constant $C$ such that $X\leq CY.$ We will also use the notation $X\lesssim_{q} Y$, which means there exists a positive constant $C(q)$ depending on $q$ only such that $X\leq C(q)Y$. Sometimes in this paper, we will write $X\lesssim Y\leq Z$. This means that there exists a positive universal constant $C$ such that $X\leq CY\leq Z.$ In this paper we will sometimes refer to a quantity (for example $\mathcal{M}$) being `sufficiently large', which should be understood as $\mathcal{M}$ being larger than some universal constant that can (in principle) be specified.

For a vector $a$, $a_{i}$ denotes the $i^{th}$ component of $a$. For $(x,t)\in\mathbb{R}^4$ and $r>0$ we denote
$B(x,r):=\{y\in\mathbb{R}^3: |y-x|<r \}$ and $Q((x,t),r):= B(x,r)\times (t-r^2,t)$. When $(x,t)=(0,0)$, we will sometimes write $Q(0,r):=Q((0,0),r).$ We denote the average of a function $f$ over the ball $B(x,r)$ by $$(f)_{B(x,r)}:=\frac{1}{|B(0,r)|}\int\limits_{B(x,r)} f dy.$$ Here, $|B(0,r)|$ denotes the Lebesgue measure of the ball $B(0,r)$.\\ For $a,\, b\in\R^3$, we write $(a\otimes b)_{\alpha\beta}=a_\alpha b_\beta$, and for $A,\, B\in M_3(\R)$, $A:B=A_{\alpha\beta}B_{\alpha\beta}$. Here and in the whole paper we use Einstein's convention on repeated indices.  For $F:\Omega\subseteq\mathbb{R}^3\rightarrow\mathbb{R}^3$, we define $\nabla F\in M_{3}(\mathbb{R})$ by $(\nabla F(x))_{\alpha\beta}:= \partial_{\beta} F_{\alpha}$. 

For ${\lambda}\in \mathbb{R}$, $\floor*{\lambda}$ denotes the greatest integer less than $\lambda$. Furthermore, $\ceil*{\lambda}$ denotes the smallest integer greater than $\lambda$. 
\subsubsection{Sobolev Spaces, Bochner Spaces and Norms}\label{sobbochnormnotation}
For $\mathcal{O}\subseteq \mathbb{R}^3$, $k\in \mathbb{N}$ and $1\leq p\leq \infty$, $W^{k,p}(\mathcal{O})$ denotes the Sobolev space with differentiability $k$ and integrability $p$.
  
If $X$ is a Banach space with norm $\|\cdot\|_{X}$, then $L^{s}(a,b;X)$, with $a<b$ and $s\in[1,\infty)$,  will denote the usual Banach space of strongly measurable $X$-valued functions $f(t)$ on $(a,b)$ such that
$$\|f\|_{L^{s}(a,b;X)}:=\left(\int\limits_{a}^{b}\|f(t)\|_{X}^{s}dt\right)^{\frac{1}{s}}<+\infty.$$ 
The usual modification is made if $s=\infty$. Let $C([a,b]; X)$ denote the space of continuous $X$ valued functions on $[a,b]$ with usual norm. In addition, let $C_{w}([a,b]; X)$ denote the space of $X$ valued functions, which are continuous from $[a,b]$ to the weak topology of $X$. 

Let $\mathcal{O}\subseteq\mathbb{R}^n$ and $a,b\in\mathbb{R}$. Sometimes we will denote $L^{p}(a,b; L^{q}(\mathcal{O}))$ and $L^{p}(a,b; W^{k,q}(\mathcal{O}))$ by $L^{p}_{t}L^{q}_{x}(\mathcal{O}\times (a,b))$ and $L^{p}_{t}W^{k,q}_{x}(\mathcal{O}\times (a,b)).$ For the case $p=q$ we will often write $L^{p}_{x,t}(\mathcal{O}\times (a,b))$ to denote $L^{p}(a,b; L^{p}(\mathcal{O})$. For convenience, we will often use the following notation for energy type norms
\begin{equation}\label{energynormdef}
\|f\|_{L^{\infty}_{t}L^{2}_{x}\cap L^{2}_{t}\dot{H}^{1}(\mathcal{O}\times (a,b))}:=\Big(\|f\|_{L^{\infty}(a,b; L^{2}(\mathcal{O}))}^2+\|\nabla f\|^2_{L^{2}(a,b; L^{2}(\mathcal{O}))}\Big)^{\frac{1}{2}}.
\end{equation}
\subsubsection{Finite Energy Solutions and Suitable Weak Solutions}\label{finiteneergyswsnotation}
We say that $v$ is a finite energy solution on $(T_1,T)$ if
$v\in C_{w}([T_1,T]; L_{\sigma}^{2}(\mathbb{R}^3)\cap L^{2}(T_1,T; \dot{H}^{1}(\mathbb{R}^3))$ is a distributional solution to \eqref{NSEintro} and satisfies the \textit{global energy inequality}
\begin{equation}\label{globalenergyinequalityintro}
\|v(\cdot,t)\|_{L^{2}(\mathbb{R}^3)}^2+\int\limits_{T_1}^{t}\int\limits_{\mathbb{R}^3} |\nabla v(x,s)|^2 dxds\leq \|v(\cdot,T_1)\|_{L^{2}(\mathbb{R}^3))}^2\quad\forall t\in [T_1,T].
\end{equation} Here,
$L^{2}_{\sigma}(\mathbb{R}^3):=\overline{\{\phi\in C^{\infty}_{0}(\mathbb{R}^3):\textrm{div}\,\phi =0\}}. $

  Let $\mathcal{O}\subseteq\mathbb{R}^3$. We say that $(v,p)$ is a \textit{suitable weak solution} to the Navier-Stokes equations in $\mathcal{O}\times (T_1,T)$ if it fulfills the properties described in \cite[Definition 6.1]{gregory2014lecture}.
One such property we will refer to in this manuscript is the \textit{local energy inequality} for such $(v,p)$.  In particular, $(v,p)\in C_{w}([T_1,T];L^{2}_{x}(\mathcal{O}))\cap L^{2}_{t}\dot{H}^{1}_{x}(\mathcal{O}\times (T_1,T))\times L^{\frac{3}{2}}_{x,t}(\mathcal{O}\times (T_1,T))$ is said to satisfy the local energy inequality in $\mathcal{O}\times (T_1,T)$ if  the inequality
\begin{align}\label{e.lei}
\begin{split}
\int\limits_{\mathcal{O}}|v(x,t)|^2\phi(x,t) dx+2\int\limits_{-1}^t\int\limits_{\mathcal{O}}|\nabla v|^2\phi dxds
\leq \int\limits_{-1}^t\int\limits_{\mathcal{O}}|v|^2(\partial_t\phi+\Delta\phi)+(|v|^2+2p)v\cdot\nabla\phi dxds
\end{split}
\end{align}
holds for almost every $t\in(T_1,T]$ and all $0\leq \phi\in C^\infty_{0}(\mathcal{O}\times (T_1,\infty))$.
\subsubsection{Singular Points and Criticality}\label{singpointsnotation}
Let $\mathcal{O}\subseteq\mathbb{R}^3$. Let $(v,p)$ be a suitable weak solution to the Navier-Stokes equations on $\mathcal{O}\times (T_1,T)$. We say that $z_0=(x_0,t_0)\in \overline{\mathcal{O}}\times (T_1,T]$ is a \textit{singular point} if
\begin{equation}\label{leraysingpoint}
v\notin L^{\infty}_{x,t}((B(x_0,r)\cap\mathcal{O})\times (t_0-r^2, t_0))
\end{equation}
for all $r>0$ sufficiently small.

Consider the Navier-Stokes equations with forcing $f$
\begin{equation}\label{NSEprelim}
\partial_{t}v-\Delta v+v\cdot\nabla v+\nabla p=f,\quad\textrm{div}\,v=0.
\end{equation}
 This is invariant with respect to the scaling
\begin{equation}\label{NSErescalingforcing}
(v_{\lambda}(x,t),p_{\lambda}(x,s), v_{0\lambda}(x)):=(\lambda v(\lambda x,\lambda^2 t), \lambda^2p(\lambda x,\lambda^2 t), \lambda^3 f(\lambda x, \lambda^2 t))\quad\textrm{with}\,\,\lambda>0.
\end{equation} 
We say that the scalar quantity $G(v,p,f)$ is \textit{critical} or \textit{scale-invariant} if $G(v_{\lambda}, p_{\lambda}, f_{\lambda})= G(v,p,f)$ for all $\lambda>0.$ The quantity $G(v,p,f)$ is said to be \textit{supercritical} if there exists an $\alpha>0$ such that $G(v_{\lambda}, p_{\lambda}, f_{\lambda})= \lambda^{-\alpha}G(v,p,f)$ for all $\lambda>0.$ Finally, $G(v,p,f)$ is said to be \textit{subcritical} if there exists an $\beta>0$ such that $G(v_{\lambda}, p_{\lambda}, f_{\lambda})= \lambda^{\beta}G(v,p,f)$ for all $\lambda>0.$
We will sometimes use phrases such as `$v/f$ belongs to critical/supercritical/subcritical spaces'. This means that $v/f$ belongs to a space-time Lebesgue space, whose norm is critical/supercritical/subcritical.
\subsection{Bogovskiĭ Operator on an Annulus}
\begin{lemma}(Bogovskiĭ operator on an annulus)
\label{lem:bogovskii} 
	Let $R\geq 1$ and define the annulus\\ $A(R):=\{x\in\mathbb{R}^3: R<|x|<2R \}.$

There exists a linear operator $B: C^\infty_{0,\avg}(A(R)) \to C^\infty_0(A(R))$
	satisfying (denote $w = Bg$) the equation
	\begin{equation}
	\label{eq:divprob}
	\left\lbrace
	\begin{aligned}
	\Div\, w &= g & &\text{ in } A(R) \\
	w \big|_{\p A(R)} &= 0 & &\text{ on } \partial A(R).
	\end{aligned}
	\right.
	\end{equation}
	Here and in the sequel, $\avg$ denotes zero spatial average. 

	Let $k \in \N_0$ and $1 < p < \infty$. Then, for all $g \in C^\infty_{0,\avg}(A(R))$,
	\begin{equation}
	\label{timeindepbogest}
	\norm{\nabla^{k+1} w}_{L^{p}(A(R))} \leq C(k,p) \norm{g}_{W^{k,p}(A(R))}.
	\end{equation}
	 Hence, $B$ extends uniquely to a bounded linear operator
	$B : \mathring{W}^{k,p}_{\avg}(A(R)) \to \mathring{W}^{k+1,p}(A(R))$
	solving~\eqref{eq:divprob}, where $\mathring{}$ denotes the closure of test functions.

	Let $I \subset \R$ be an open interval and $g \in L_1(I;L_{p,\avg}(A(R)))$.
	Consider the linear operator $B$ defined by applying the above operator at almost every time.
	
	If $\p_t g \in L_1(I;L_p(A(R)))$, then $B$ commutes with the time derivative:
	\begin{equation}
		\label{Bcommutewithpt}
		\p_t B(g) = B(\p_t g).
	\end{equation}
\end{lemma}
For the time-independent assertions, see~\cite{bogovskii} and \cite[Theorem III.3.3]{galdi}.
The fact that the constant in \eqref{timeindepbogest} is independent of $R\geq 1$ can be seen by rescaling $w_{R}(x)=w(Rx)$ and $g_{R}(x)=Rg(Rx)$ and  then applying the Bogovskiĭ operator on $A(1)$.
 To prove~\eqref{Bcommutewithpt}, one may use the finite difference operator $D^h_t \varphi = \left( \varphi + \varphi(\cdot+h) \right)/h$ and take $h \to 0^+$. See \cite[ Exercise III.3.7]{galdi}. 
\section{Local-in-space smoothing and localized backward propagation of vorticity concentration}
\subsection{Local-in-Space Smoothing for Suitable Weak Solutions}
In \cite{barkerprange2020},  Prange and the author used a strategy to produce quantitative estimates based on (i) local-in-space smoothing for a short time \cite[Theorem 5.1]{ barkerprange2020}, (ii) backward propagation of vorticity concentration \cite[Lemma 3.1]{barkerprange2020}. In \cite{barkerprange2020}, these concepts and statements are formulated in a global context (i.e the velocity field is defined on $\mathbb{R}^3$ for each time). In order to prove the localized quantitative estimate in Theorem \ref{locest}, we require analogues of these concepts but in a local context rather than a global context. 

It is already known that local-in-space short time smoothing is essentially a local phenomenon, that applies is a local context (see \cite[Theorem 3]{barkerprange2020} and \cite[Theorem 1.1]{KMT21IMRN}). We use Theorem 1.1, Remark 1.2 and Theorem 3.1 of \cite{KMT21}, which are formulated in a convenient way for our purposes. Here is the statement below, which is essentially taken from \cite{KMT21}.
\begin{pro}(Local-in-space short time smoothing, \cite[Theorem 1.1, Remark 1.2 and Theorem 3.1]{KMT21})\label{KMT}\\
There exists  positive universal constants $c_0$ and $\epsilon_{*}$ such that the following holds true.\\ Let $M\geq 1$ and define $T_0;=c_0M^{-{18}}.$ Suppose that $(v,p)$ is a suitable weak solution to the Navier-Stokes equations on $B(0,4)\times (0,T_0)$ such that
\begin{equation}\label{vi.dKMT}
\lim_{t\rightarrow 0^+}\|v(\cdot,t)-v_0\|_{L^{2}(B(0,4))}=0,\quad\textrm{with}\,\, v_0\in L^{2}(B(0,4)).
\end{equation}
Furthermore suppose that
\begin{equation}\label{venergyMKMT}
\|v\|^{2}_{L^{\infty}_{t}L^{2}_{x}\cap L^{2}_{t}\dot{H}^{1}( B(0,4)\times (0,T_0))}+\|p\|_{L^{\frac{3}{2}}_{x,t}(B(0,4)\times (0,T_0))}\leq M
\end{equation}
and
\begin{equation}\label{L3smallKMT}
\|v_0\|_{L^{3}(B(0,3))}\leq \epsilon_*.
\end{equation}
Then the above hypothesis imply that $v$ 
 satisfies the quantitative bound
\begin{equation}\label{quantboundedKMT}
|\nabla^{j}v(x,t)|\lesssim_{j} t^{-\frac{j+1}{2}}\quad\textrm{where}\quad j=0,1\ldots\quad\textrm{and}\quad (x,t)\in B(0,2)\times (0, T_0).
\end{equation}
\end{pro}
\begin{remark}\label{KMThigher}
Note that in \cite{KMT21}, \eqref{quantboundedKMT} is stated and proven for $j=0$ by showing that for certain $r>0$ and $x_0$, we have
$$\int\limits_{0}^{r^2}\int\limits_{B(x_0,r)} |v|^3+|p|^{\frac{3}{2}} dxdt\leq \varepsilon_{CKN} $$
and then applying Caffarelli, Kohn and Nirenberg's celebrated $\varepsilon$-regularity criteria \cite{CKN}. It is well known that the $\varepsilon$- regularity criteria in \cite{CKN} gives higher derivative estimates via parabolic smoothing. Thus the proof  of \cite[Theorem 1.1]{KMT21} also gives \eqref{quantboundedKMT} for all $j\in\mathbb{N}.$
\end{remark}
\begin{remark}\label{KMTsmallertime}
From \cite[Theorem 3.1]{KMT21} the hypothesis that $(v,p)$ is a suitable weak solution on $B(0,4)\times (0,T_0)$ can be replaced by the hypothesis that $v$ is a suitable weak solution on $B(0,4)\times (0,T)$ for any given $T>0$. Under this change of hypothesis, the estimate \eqref{quantboundedKMT} instead holds for $(x,t)\in B(0,2)\times (0, \min(T,T_0))$.
\end{remark}

\begin{remark}[Removing the pressure]\label{withoutpressure}
From \cite[Theorem 16]{Kwon}, a version of Proposition \cite{KMT21} is stated without any assumed bound on the pressure. Such a statement does not obviously improve our main results, due to the appearance of the pressure in the quantitative truncation procedure (Proposition \ref{truncationprocedure}).
\end{remark}

Now we state a corollary, which is analogous to the global result of \cite[Theorem 5.1]{barkerprange2020} (which in turn is based upon ideas in \cite{jiasverak}) but in the local setting. 
\begin{cor}{(Local-in-space short time smoothing, with initial data locally in $L^{6}$)}\label{locinspaceL6}\\There exists a  positive universal constant $c_1$  such that the following holds true.\\Let $M$ and $N$ be sufficiently large and define $T_1;=c_1M^{-{18}}N^{-52}.$ Suppose that $(v,p)$ is a suitable weak solution to the Navier-Stokes equations on $B(0,4)\times (0,T_1)$ such that
\begin{equation}\label{vi.dL6}
\lim_{t\rightarrow 0^+}\|v(\cdot,t)-v_0\|_{L^{2}(B(0,4))}=0,\quad\textrm{with}\,\, v_0\in L^{2}(B(0,4)).
\end{equation}
Furthermore suppose that
\begin{equation}\label{venergyL6}
\|v\|^{2}_{L^{\infty}_{t}L^{2}_{x}\cap L^{2}_{t}\dot{H}^{1}(B(0,4)\times (0,T_1))}+\|p\|_{L^{\frac{3}{2}}_{x,t}(B(0,4)\times (0,T_1))}\leq M
\end{equation}
and
\begin{equation}\label{L3smallL6}
\|v_0\|_{L^{6}(B(0,3))}\leq N.
\end{equation}
Then the above hypothesis imply that $v$
 satisfies the quantitative bound
\begin{equation}\label{quantboundedL6}
|\nabla^{j}v(x,t)|\lesssim_{j} t^{-\frac{j+1}{2}}\quad\textrm{where}\quad j=0,1\ldots\quad\textrm{and}\quad (x,t)\in B(0,2)\times (0, T_1).
\end{equation}
\end{cor}
\begin{proof}
Take $\mu\in (0,\tfrac{1}{3}]$, which will be fixed later.
Throughout we fix $x_0\in B(0,2\mu^{-1})$. Note that
\begin{equation}\label{ballinclusions}
B(x_0,3)\subseteq B(0,3\mu^{-1})\quad\textrm{and}\quad B(x_0,4)\subseteq B(0,4\mu^{-1}).
\end{equation}
Now let us consider the rescaled functions $v_{\mu}:B(0,4\mu^{-1})\times (0,T_1\mu^{-2})\rightarrow \mathbb{R}^3\,\,\textrm{and}\,\,p_{\mu}:B(0,4\mu^{-1})\times (0,T_1\mu^{-2})\rightarrow \mathbb{R}$ defined by
\begin{equation}\label{rescalevpmu}
(v_{\mu}(x,t), p_{\mu}(x,t)):=(\mu v(\mu x, \mu^2 t), \mu^2 p(\mu x, \mu^2 t)).
\end{equation}
Let us also define $v_{\mu0}:B(0,4\mu^{-1})\rightarrow \mathbb{R}^3$ by
\begin{equation}\label{rescaleid}
v_{\mu0}(x):=\mu v_{0}(\mu x).
\end{equation}
Using the properties \eqref{venergyL6}-\eqref{L3smallL6}, we infer that
\begin{equation}\label{vmupmunorm}
\|v_{\mu}\|^{2}_{L^{\infty}_{t}L^{2}_{x}\cap L^{2}_{t}\dot{H}^{1}(B(x_0,4)\times (0,T_1\mu^{-2}))}+\|p_{\mu}\|_{L^{\frac{3}{2}}_{x,t}(B(x_0,4)\times (0,T_1\mu^{-2}))}\leq \hat{M}:=3\mu^{-\frac{4}{3}}M\quad\textrm{and}
\end{equation}
\begin{equation}\label{vidmu}
\|v_{\mu0}\|_{L^{3}(B(x_0,3))}\leq C_{leb} N\mu^{\frac{1}{2}}.
\end{equation}
Here, $C_{leb}\in (0,\infty)$ is a universal constant coming from the continuous embedding $L^{6}(B(x_0,3))\hookrightarrow L^{3}(B(x_0,3)).$ In view of  \eqref{vidmu}, we now take 
\begin{equation}\label{mudefined}
\mu:=\min\Big\{\frac{\epsilon_{*}^2}{C_{leb}^2N^2},\frac{1}{3}\Big\}=\frac{\epsilon_{*}^2}{C_{leb}^2N^2}
\end{equation}
 for $N$ sufficiently large. This gives that
\begin{equation}\label{vmuL3small}
\|v_{\mu0}\|_{L^{3}(B(x_0,3))}\leq \epsilon_{*}\,\,\textrm{and}\,\, \hat{M}=C_{univ}N^{\frac{8}{3}}M.
\end{equation}
Now we take $c_1$ in Corollary \ref{locinspaceL6} to be an appropriate universal constant such that
\begin{equation}\label{T1def}
T_1\mu^{-2}=c_0 \hat{M}^{-18}.
\end{equation} Here, $c_0$ is as in Proposition \ref{KMT}. Then \eqref{vmuL3small}-\eqref{T1def} and a translation in space allow us to apply Proposition \ref {KMT}  to get  
$$|\nabla^{j}v_{\mu}(x,t)|\lesssim_{j}t^{-\frac{j+1}{2}}\quad\textrm{where}\quad j=0,1\ldots\quad\textrm{and}\quad (x,t)\in B(x_0,2)\times (0, T_1\mu^{-2}). $$ Recalling that $x_0\in B(0,2\mu^{-1})$ is arbitrary, we undo the Navier-Stokes rescaling \eqref{rescalevpmu} to finally infer that
$$ |\nabla^{j}v(x,t)|\lesssim_{j}t^{-\frac{j+1}{2}}\quad\textrm{where}\quad j=0,1\ldots\quad\textrm{and}\quad (x,t)\in B(0,2)\times (0, T_1).$$ 
 
\end{proof}
\begin{remark}\label{locaL6smallertime}
It follows from Remark \ref{KMTsmallertime} and the above proof that the hypothesis that $(v,p)$ is a suitable weak solution on $B(0,4)\times (0,T_1)$ can be replaced by the hypothesis that $v$ is a suitable weak solution on $B(0,4)\times (0,T)$ for any given $T>0$. Under this change of hypothesis, the estimate \eqref{quantboundedL6} instead holds for $(x,t)\in B(0,2)\times (0, \min(T,T_1))$.
\end{remark}
Let us now state another corollary to Proposition \ref{KMT}, which concerns local-in-space smoothing of solutions with locally square integrable vorticity.
\begin{cor}{(Local-in-space short time smoothing, with initial vorticity locally in $L^2$)}\label{localinspacevort}\\There exists a  positive universal constant $c_2$  such that the following holds true.\\Let $M$ be sufficiently large and define $T_2;=c_2M^{-{44}}.$ Suppose that $(v,p)$ is a suitable weak solution to the Navier-Stokes equations on $B(0,4)\times (0,T_2)$ such that
\begin{equation}\label{vi.dvort}
\lim_{t\rightarrow 0^+}\|v(\cdot,t)-v_0\|_{L^{2}(B(0,4))}=0,\quad\textrm{with}\,\, v_0\in W^{1,2}(B(0,4)).
\end{equation}
Furthermore suppose that
\begin{equation}\label{venergyvort}
\|v\|^{2}_{L^{\infty}_{t}L^{2}_{x}\cap L^{2}_{t}\dot{H}^{1}(B(0,4)\times (0,T_2))}+\|p\|_{L^{\frac{3}{2}}_{x,t}(B(0,4)\times (0,T_2))}\leq M
\end{equation}
and that for $\omega_0=\nabla\times v_0$
\begin{equation}\label{vorti.d}
\|\omega_0\|_{L^{2}(B(0,\frac{7}{2}))}^2\leq M.
\end{equation}
Then the above hypothesis imply that $v$
  satisfies the quantitative bound
\begin{equation}\label{quantboundedvort}
|\nabla^{j}v(x,t)|\lesssim_{j} t^{-\frac{j+1}{2}}\quad\textrm{where}\quad j=0,1\ldots\quad\textrm{and}\quad (x,t)\in B(0,2)\times (0, T_2).
\end{equation}
\end{cor}
\begin{proof}
Recall that from the Biot-Savart law $-\Delta v_0=\nabla\times \omega_0.$ Using this and elliptic regularity theory gives
$$\|\nabla v_0\|^{2}_{L^{2}(B(0,3))}\lesssim \|\omega_0\|_{L^{2}(B(0,\frac{7}{2}))}^2+\|v_0\|_{L^{2}(B(0,\frac{7}{2}))}^2. $$
Applying the Sobolev embedding theorem then gives
$$\|v_0\|_{L^{6}(B(0,3))}\lesssim \|\nabla v_0\|_{L^{2}(B(0,3))}+\| v_0\|_{L^{2}(B(0,3))}\lesssim \|\omega_0\|_{L^{2}(B(0,\frac{7}{2}))}+\|v_0\|_{L^{2}(B(0,\frac{7}{2}))}.
$$
Using this and \eqref{venergyvort}-\eqref{vorti.d} we get that 
\begin{equation}\label{vortimpliesL6}
\|v_0\|_{L^{6}(B(0,3))}\leq C_{univ}M^{\frac{1}{2}}.
\end{equation} 
Here, $C_{univ}$ is a positive universal constant.
Taking $N=C_{univ}M^{\frac{1}{2}}$ and taking $c_2$ to be an appropriate universal constant, we get the desired conclusion by directly applying Corollary \ref{locinspaceL6}.
\end{proof} 
\begin{remark}\label{locavortsmallertime}
As was the case in Remark \ref{locaL6smallertime}, the hypothesis that $(v,p)$ is a suitable weak solution on $B(0,4)\times (0,T_2)$ can be replaced by the hypothesis that $v$ is a suitable weak solution on $B(0,4)\times (0,T)$ for any given $T>0$. Under this change of hypothesis, the estimate \eqref{quantboundedvort} instead holds for $(x,t)\in B(0,2)\times (0, \min(T,T_2))$.
\end{remark}
\subsection{Local behavior of Vorticity Implies Localized Quantitative Estimates}
The strategy to prove the local estimates in Theorem \ref{locest} is a localized version of the strategy in \cite{barkerprange2020}. As in \cite{barkerprange2020}, we rely on connecting the behavior of the vorticity on spatial scales with quantitative estimates of the velocity field (see \cite[Lemma 3.1-3.2]{barkerprange2020}). However unlike in \cite{barkerprange2020}, in order to prove Theorem \ref{locest} we require such a connection to be in a purely local setting. This is given by the following proposition.
\begin{pro}\label{vortlocrescale}
Let $M$ be sufficiently large.
Suppose that $(v,p)$ is a smooth solution to the Navier-Stokes equations on $B(0,1)\times (-1,0]$ satisfying 
\begin{equation}\label{vscaleinvariantvort}
\sup_{0<r\leq 1}\Big\{\frac{1}{r}(\|v\|_{L^{\infty}_{t}L^{2}_{x}\cap L^{2}_{t}\dot{H}^{1}(Q(0,r))}^2+\frac{1}{r^{\frac{4}{3}}}\|p\|_{L^{\frac{3}{2}}_{x,t}(Q(0,r))}\Big\}\leq M.
\end{equation}
Furthermore, suppose that there exists $t\in (-M^{-48},0)$ such that the vorticity $\omega=\nabla\times v$ satisfies
\begin{equation}\label{vortsmallrescale}
\int\limits_{B(0,M^{24}(-t)^{\frac{1}{2}})} |\omega(x,t)|^2 dx\leq \frac{M^{-22}}{({-t})^{\frac{1}{2}}}.
\end{equation}
Then the above assumptions imply that
\begin{equation}\label{quantestvortrescale}
|\nabla^{j}v(x,s)|\lesssim_{j}(s-t)^{-\frac{j+1}{2}}\quad\textrm{where}\quad j=0,1\ldots\quad\textrm{and}\quad (x,s)\in B(0,M^{23}({-t})^{\frac{1}{2}})\times (t, 0].
\end{equation}
\end{pro} 
\begin{proof}
First note that \eqref{vscaleinvariantvort} implies that for $M$ sufficiently large
\begin{equation}\label{vscalevortatt}
\frac{1}{4M^{23}({-t})^{\frac{1}{2}}}\|v\|_{L^{\infty}_{t}L^{2}_{x}\cap L^{2}_{t}\dot{H}^{1}(Q(0,4M^{23}({-t})^{\frac{1}{2}}))}^2+\frac{1}{(4M^{23}({-t})^{\frac{1}{2}})^{\frac{4}{3}}}\|p\|_{L^{\frac{3}{2}}_{x,t}(Q(0,4M^{23}({-t})^{\frac{1}{2}}))}\leq M.
\end{equation}
Take $\lambda:=M^{23}({-t})^{\frac{1}{2}}$ and let us consider the rescaled functions $v_{\lambda}:B(0,\lambda^{-1})\times (0,M^{-46})\rightarrow \mathbb{R}^3\,\,\textrm{and}\,\,p_{\lambda}:B(0,\lambda^{-1})\times (0,M^{-46})\rightarrow \mathbb{R}$ defined by
\begin{equation}\label{rescalevplambda}
(v_{\lambda}(y,s), p_{\lambda}(y,s)):=(\lambda v(\lambda y, \lambda^2 s+t), \lambda^2 p(\lambda y, \lambda^2 s+t)).
\end{equation}
Under this rescaling $B(0,\lambda^{-1})\supset B(0,4)$ and furthermore \eqref{vortsmallrescale} and \eqref{vscalevortatt} imply that for $M$ sufficiently large we have
\begin{equation}\label{vlambdaplambdabound}
\|v_{\lambda}\|^{2}_{L^{\infty}_{t}L^{2}_{x}\cap L^{2}_{t}\dot{H}^{1}(B(0,4)\times (0,M^{-46}))}+\|p_{\lambda}\|_{L^{\frac{3}{2}}_{x,t}(B(0,4)\times (0,M^{-46}))}\leq 4^{\frac{4}{3}}M
\end{equation}
and
\begin{equation}\label{omegalambdainitial}
\int\limits_{B(0,\tfrac{7}{2})} |\omega_{\lambda}(y,0)|^2 dy\leq M\leq 4^{\frac{4}{3}}M.
\end{equation}
Now, \eqref{vlambdaplambdabound}-\eqref{omegalambdainitial} allows us to apply Corollary \ref{localinspacevort} and Remark \ref{locavortsmallertime} with $M$ replaced by $\hat{M}:=4^{\frac{4}{3}}M$. If $M$ is sufficiently large this gives that 
$$|\nabla^{j}v_{\lambda}(y,s)|\lesssim_{j}s^{-\frac{j+1}{2}}\quad\textrm{where}\quad j=1\ldots\quad\textrm{and}\quad (y,s)\in B(0,2)\times (0, M^{-46}].$$ We then undo the Navier-Stokes rescaling \eqref{rescalevplambda} to finally infer the desired conclusion \eqref{quantestvortrescale} for the original functions $(v,p)$.
\end{proof}
\subsection{Localized Backward Propagation of Vorticity Concentration}
As previously mentioned, backward propagation plays a central role in the strategy to produce quantitative estimates in \cite{barkerprange2020}. The relevant statement is \cite[Lemma 3.1]{barkerprange2020}, which is in a global context. Now, we state and prove a localized version of \cite[Lemma 3.1]{barkerprange2020}, which will play a crucial role in proving Theorem \ref{locest}.
\begin{pro}\label{backpropvortloc}
Let $M$ be sufficiently large.
Suppose that $(v,p)$ is a smooth solution to the Navier-Stokes equations on $B(0,1)\times (-1,0]$ satisfying 
\begin{equation}\label{vscaleinvariantvortprop}
\sup_{0<r\leq 1}\Big\{\frac{1}{r}\|v\|_{L^{\infty}_{t}L^{2}_{x}\cap L^{2}_{t}\dot{H}^{1}(Q(0,r))}^2+\frac{1}{r^{\frac{4}{3}}}\|p\|_{L^{\frac{3}{2}}_{x,t}(Q(0,r))}\Big\}\leq M.
\end{equation}
Furthermore, suppose that $t$ and $t'\in (-M^{-48},0)$ are such that
\begin{equation}\label{timeswellseparated}
\frac{-t'}{-t}\leq M^{-50}
\end{equation}
and 
\begin{equation}\label{vortconct'}
\int\limits_{B(0,M^{24}(-t')^{\frac{1}{2}})} |\omega(x,t')|^2 dx>  \frac{M^{-22}}{(-t')^{\frac{1}{2}}}.
\end{equation}
Then the above assumptions imply
\begin{equation}\label{vortconct}
\int\limits_{B(0,M^{24}(-t)^{\frac{1}{2}})} |\omega(x,t)|^2 dx> \frac{M^{-22}}{(-t)^{\frac{1}{2}}}.
\end{equation}
\end{pro} 
\begin{proof}
Assume that the contrapositive to \eqref{vortconct} holds. Then by Proposition \ref{vortlocrescale}, we have that there exists a positive universal constant $C_{univ}$ such that
\begin{equation}\label{vortbddrescale}
|\omega (x,s)|\leq \frac{C_{univ}}{-t}\quad\textrm{where}\quad (x,s)\in B(0,M^{23}({-t})^{\frac{1}{2}})\times (\tfrac{t}{2},0].
\end{equation}
Note that the assumption \eqref{timeswellseparated} implies that $$B(0, M^{24}({-t'})^{\frac{1}{2}})\times \{t'\}\subset B(0,M^{23}({-t})^{\frac{1}{2}})\times (\tfrac{t}{2},0].$$
This together with \eqref{vortbddrescale} implies that for $M$ sufficiently large we have 
\begin{equation}\label{vortL2bddt'}
\int\limits_{B(0,M^{24}({-t'})^{\frac{1}{2}})} |\omega(x,t')|^2 dx\leq \frac{ M^{74}}{({-t'})^{\frac{1}{2}}}\Big(\frac{-t'}{-t}\Big)^2.
\end{equation}
Now, using the assumption \eqref{timeswellseparated} gives that
$$\int\limits_{B(0,M^{24}({-t'})^{\frac{1}{2}})} |\omega(x,t')|^2 dx\leq \frac{ M^{-26}}{({-t'})^{\frac{1}{2}}}, $$
which contradicts \eqref{vortconct'}.
\end{proof}
\section{Quantitative Truncation Procedure}
\subsection{Local Scale-Invariant Morrey Bounds}
A key part of showing Theorem \ref{locest} is showing that the solution in Theorem \ref{locest} can be truncated to produce a solution of the Navier-Stokes equations with forcing. It is crucial that the forcing is compactly supported and that the subcritical norms of the forcing are quantified and sufficiently small. A central part of this quantitative truncation procedure involves using that the solution in Theorem \ref{locest} has a scale-invariant Morrey bound, which we demonstrate in this subsection. Before doing so, let us introduce some relevant notation.

If $(v,p)$ is a suitable weak solution on $Q(0,4)$, then we define the following quantities for $\rho\in (0,4]$. Namely,
\begin{equation}\label{rescaledvelocity}
A(v,\rho):=\frac{1}{\rho}\|v\|_{L^{\infty}_{t}L^{2}_{x}(Q(0,\rho))}^2,\quad E(v,\rho):=\frac{1}{\rho}\|\nabla v\|^{2}_{L^{2}(Q(0,\rho))},\quad C(v,\rho)=\frac{1}{\rho^2}\|v\|^3_{L^{3}(Q(0,\rho))},
\end{equation} 

\begin{equation}\label{rescaledpressure}
D(p,\rho):=\frac{1}{\rho^2}\|p\|_{L^{\frac{3}{2}}_{x,t}(Q(0,\rho))}^{\frac{3}{2}}\quad\textrm{and}\quad \mathcal{E}(v,p,\rho):=A(v,\rho)+E(v,\rho)+D(p,\rho).
\end{equation}

\begin{lemma}\label{Morreybound}
Suppose that $(v,p)$ is a suitable weak solution to the Navier-Stokes equations on $Q(0,4)$. Furthermore, suppose that for some $\mathcal{M}\geq 1$ we have
\begin{equation}\label{vL3Morrey}
\|v\|_{L^{\infty}_{t}L^{3}_{x}(Q(0,4))}\leq \mathcal{M}.
\end{equation}
Then the above assumptions imply that
\begin{equation}\label{Morreyboundquant}
\sup_{z_0=(x_0,t_0)\in B(0,1)\times (-1,0],\,\, 0<r\leq 1}\left\{\frac{1}{r^2}\int\limits_{Q(z_0,r)} |p|^{\frac{3}{2}}dxdt+\frac{1}{r}\|v\|_{L^{\infty}_{t}L^{2}_{x}\cap L^{2}_{t}\dot{H}^{1}(Q(z_0,r))}^2\right\}
\lesssim \mathcal{M}^3+\|p\|^{\frac{3}{2}}_{L^{\frac{3}{2}}_{x,t}(Q(0,4))}. 
\end{equation}
\end{lemma}
\begin{proof}
First observe that by using spatial translations, it is  sufficient to show that the following holds true. Namely, if $(v,p)$ is a suitable weak solution on $Q(0,2)$ with 
\begin{equation}\label{vL3morreyreduce}
\|v\|_{L^{\infty}_{t}L^{3}_{x}(Q(0,2))}\leq \mathcal{M}\,\,(\mathcal{M}\geq 1)
\end{equation}
then 
\begin{align}\label{morreyboundquantreduce}
\begin{split}
\sup_{ 0<r\leq 1}\mathcal{E}(v,p,r)
\lesssim \mathcal{M}^3+\|p\|^{\frac{3}{2}}_{L^{\frac{3}{2}}_{x,t}(Q(0,2))}.
\end{split} 
\end{align}
From \eqref{vL3morreyreduce} and the local energy inequality \eqref{e.lei}, we observe that in order to show \eqref{morreyboundquantreduce} it is sufficient to prove that
\begin{equation}\label{morreyboundpresreduce}
\sup_{0<r\leq 1} D(p,r)\lesssim \mathcal{M}^3+\|p\|^{\frac{3}{2}}_{L^{\frac{3}{2}}_{x,t}(Q(0,2))}.
\end{equation} 
From \cite[pg. 847 (47)]{greogryvladimirhandbook} we have that for any $\theta\in (0,1)$ and $0<r\leq 1$
$$D(p,\theta r)\leq C_{univ}(\theta D(p,r)+{\theta^{-2}}C(v,r)). $$
Here, $C_{univ}$ is a positive universal constant. Thus using \eqref{vL3morreyreduce} we obtain 
that for any $\theta\in (0,1)$ and $0<r\leq 1$
$$D(p,\theta r)\leq C_{univ}(\theta D(p,r)+{\theta^{-2}}\mathcal{M}^3). $$
Fixing the universal constant $\theta\in (0,1)$ such that $C_{univ}\theta\leq \frac{1}{2}$, one deduces that
\begin{equation}\label{Pestmorrey}
D(p,\theta r)\leq \frac{1}{2}D(p,r)+c(\theta)\mathcal{M}^3\quad\forall r\in (0,1].
\end{equation}
Iterating \eqref{Pestmorrey} gives that
\begin{equation}\label{Pestiterated}
D(p,\theta^k)\leq \frac{1}{2^{k}}D(p,1)+2c(\theta)\mathcal{M}^3\quad\textrm{for}\quad k=0,1,2\ldots.
\end{equation}
Take any $r\in (0,1]$. There exists $k\in\{0,1,2\ldots\}$ such that $\theta^{k+1}<r\leq \theta^{k}.$ From this, we see that
$$D(p,r)\leq\frac{1}{\theta^2}D(p,\theta^k)\leq \frac{1}{\theta^2}\Big(\frac{1}{2^{k}}D(p,1)+2c(\theta)\mathcal{M}^3\Big). $$
Hence, $$\sup_{0<r\leq 1}D(p,r)\lesssim D(p,1)+\mathcal{M}^3$$
as required.
\end{proof}

\subsection{ Quantitative Truncation Procedure: Main Statement and Proof}
A crucial part in proving Theorem \ref{locest} is the quantification of a qualitative localization procedure introduced in \cite{neustupapenel}, when the solution has an assumed scale-invariant bound. This is given by the proposition below. 
\begin{pro}\label{truncationprocedure}
 Suppose that $(v,p)$ is a smooth solution to the Navier-Stokes equations on $B(0,4)\times (-16,0]$.\\
Furthermore, suppose that $\mathcal{M}\in (0,\infty)$ is such that
\begin{equation}\label{LinfinityL3truncated}
 \|v\|_{L^{\infty}_{t}L^{3}_{x}(Q(0,4))}\leq\mathcal{M}\quad\textrm{and}
 \end{equation}
 \begin{equation}\label{Mlargetruncated}
 \max\{\mathcal{M}_{0}, \|p\|_{L^{\frac{3}{2}}_{x,t}(Q(0,4))}^{\frac{3}{2}}\}\leq \mathcal{M}. \end{equation}
Here, $\mathcal{M}_{0}$ is a sufficiently large universal constant. Then the above assumptions imply the following.\\There exists functions $V:\mathbb{R}^3\times [-2,0]\rightarrow\mathbb{R}^3$, $P:\mathbb{R}^3\times [-2,0]\rightarrow\mathbb{R}$ and $R\geq 6e^{e^{2\mathcal{M}}}$ such that
\begin{equation}\label{VP=vprescale}
(V(x,t), P(x,t))\equiv (\lambda v(\lambda x,\lambda^2 t), \lambda^2 p(\lambda x, \lambda^2 t))\quad\textrm{on}\quad B(0,R)\times [-2,0]\quad (\lambda=e^{-e^{2\mathcal{M}}}e^{-\mathcal{M}^5})
\end{equation}
and 
\begin{equation}\label{VPmorrey}
\sup_{ 0<r\leq 1}\left\{\frac{1}{r^2}\int\limits_{Q(0,r)} |P|^{\frac{3}{2}}dxdt+\frac{1}{r}\|V\|^2_{L^{\infty}_{t}L^{2}_{x}\cap L^{2}\dot{H}^{1}(Q(0,r))}\right\}
\lesssim \mathcal{M}^3.
\end{equation}
The functions $(V,P)$ are smooth on $\mathbb{R}^3\times [-2,0)$ and $B(0,R)\times [-2,0]$.
Furthermore, $(V,P)$  solves the Navier-Stokes equations with forcing $F$ on $\mathbb{R}^3\times (-2,0)$ such that for every $t\in (-2,0)$
\begin{equation}\label{forcecompactsupport}
\supp F(\cdot,t)\subseteq B(0,2R)\setminus \overline{B(0,R)}
\quad\textrm{and}
\end{equation}
\begin{equation}\label{Vcompactsupport}
\supp V(\cdot,t)\subseteq B(0,2R).
\end{equation}
Additionally $V$ can be decomposed as $V=V^{(1)}+V^{(2)}$, where $V^{(i)}:\mathbb{R}^3\times [-2,0]\rightarrow\mathbb{R}^3$ ($i\in\{1,2\}$) satisfy
\begin{equation}\label{Vdecompbound}
\|V^{(1)}\|_{L^{\infty}_{t}L^{3}_{x}(\mathbb{R}^3\times (-2,0))}=\|V\|_{L^{\infty}_{t}L^{3}_{x}(B(0,R)\times (-2,0))}\leq\mathcal{M}\quad\textrm{and}\quad \|V^{(2)}\|_{L^{\infty}_{t}L^{\frac{10}{3}}_{x}(\mathbb{R}^3\times (-2,0))}\leq e^{-e^{\mathcal{M}}}.
\end{equation}

The forcing term $F$ satisfies the estimates
\begin{equation}\label{Festimate}
\|F\|_{L^{\frac{3}{2}}_{t}L^{\infty}_{x}(\mathbb{R}^3\times (-2,0))},\,\,\|F\|_{L^{\frac{3}{2}}_{t}L^{\frac{11}{6}}_{x}(\mathbb{R}^3\times (-2,0))}\leq e^{-e^{\mathcal{M}}}.
\end{equation}
\end{pro}
\begin{proof}
\textbf{Step 1: localized quantitative regions of regularity via local-in-space smoothing}\\ First note that under the assumptions of Proposition \ref{truncationprocedure}, we can apply Lemma \ref{Morreybound} to infer that for $\mathcal{M}_{0}$ sufficiently large we have
\begin{align}\label{Morreyboundtruncation}
\begin{split}
\sup_{z_0=(x_0,t_0)\in B(0,1)\times(-1,0],\,\, 0<r\leq 1}\left\{\frac{1}{r^2}\int\limits_{Q(z_0,r)} |p|^{\frac{3}{2}}dxdt+
+\frac{1}{r}\|v\|^2_{L^{\infty}_{t}L^{2}_{x}\cap L^{2}\dot{H}^{1}(Q(z_0,r))}\right\}
\lesssim \mathcal{M}^3.
\end{split} 
\end{align}
Next, we rescale the solution $(v,p)$:
\begin{equation}\label{vprescalelambda1}
(v_{\lambda_{(1)}}(x,t),p_{\lambda_{(1)}}(x,t)):=(\lambda_{(1)}v(\lambda_{(1)}x, \lambda_{(1)}^2 t), \lambda_{(1)}^2 p(\lambda_{(1)}x,\lambda_{(1)}^2 t))\quad\textrm{with}\,\, \lambda_{(1)}:= e^{-\mathcal{M}^5}.
\end{equation}
Then $(v_{\lambda_{(1)}},p_{\lambda_{(1)}})$ is a smooth solution to the Navier-Stokes equations on $B(0,4e^{\mathcal{M}^5})\times (-16e^{2\mathcal{M}^5},0]$.
 From \eqref{LinfinityL3truncated}-\eqref{Mlargetruncated} we see that
\begin{equation}\label{vlambda1L3}
\|v_{\lambda_{(1)}}\|_{L^{\infty}_{t}L^{3}_{x}(Q(0, 4e^{\mathcal{M}^5}))}\leq \mathcal{M}^3\quad\textrm{and}
\end{equation}
\begin{equation}\label{plambda1large}
\|p_{\lambda_{(1)}}\|_{L^{\frac{3}{2}}_{x,t}(Q(0,4e^{\mathcal{M}^5}))}\leq \mathcal{M}^{\frac{2}{3}}e^{\frac{4}{3}\mathcal{M}^5}.
\end{equation} 
Furthermore, from \eqref{Morreyboundtruncation} we have
\begin{equation}\label{Morreyboundtruncationrescale}
\sup_{z_0=(x_0,t_0)\in B(0,e^{\mathcal{M}^5})\times (-e^{2\mathcal{M}^5},0],\,\, 0<r\leq e^{\mathcal{M}^5}}\left\{\frac{1}{r^2}\int\limits_{Q(z_0,r)} |p_{\lambda_{(1)}}|^{\frac{3}{2}}dxdt+\frac{1}{r}\|v_{\lambda_{(1)}}\|^2_{L^{\infty}_{t}L^{2}_{x}\cap L^{2}\dot{H}^{1}(Q(z_0,r))}\right\}\lesssim \mathcal{M}^3. 
\end{equation} 
Note that for $\mathcal{M}_{0}$ sufficiently large, we have that $4^{\mathcal{M}^4+3}\leq e^{\mathcal{M}^5}. $ Thus
$$\sum_{k=2}^{\lceil{\mathcal{M}^4\rceil}+1}\int\limits_{4^{k}<|x|\leq 4^{k+1}}|v_{\lambda_{(1)}}(x,-\mathcal{M}^{-95})|^3 dx\leq \int\limits_{B(0,e^{\mathcal{M}^5})}|v_{\lambda_{(1)}}(x,-\mathcal{M}^{-95})|^3 dx\leq\mathcal{M}^3. $$ 
Using the pigeonhole principle, there exists $k_0\in\{2,\ldots \lceil{\mathcal{M}^4\rceil}+1\}$ such that 
\begin{equation}\label{L3smallannulus}
\int\limits_{A_{k_0}}|v_{\lambda_{(1)}}(x,-\mathcal{M}^{-95})|^3 dx\leq \frac{1}{\mathcal{M}}\leq \epsilon_{*}\quad\textrm{with}\quad A_{k_0}:=\{x:4^{k_0}<|x|\leq 4^{k_0+1}\}.
\end{equation}
Here, $\mathcal{M}_{0}$ is sufficiently large and $\epsilon_{*}$ is as in \eqref{L3smallKMT} (Proposition \ref{KMT}). Now we fix any $$y_0\in\{x:4^{k_0}+4\leq|x|\leq 4^{k_0+1}-4\}\subset A_{k_0}\subset B(0,e^{\mathcal{M}^5}). $$ Then \eqref{Morreyboundtruncationrescale} implies that we have
$$\frac{1}{16}\int\limits_{-16}^{0}\int\limits_{{B(y_0,4)}} |p_{\lambda_{(1)}}|^{\frac{3}{2}} dxdt+\frac{1}{4}\|v_{\lambda_{(1)}}\|_{L^{\infty}_{t}L^{2}_{x}\cap L^{2}_{t}\dot{H}^{1}(B(y_0,4)\times (-\tfrac{1}{16},0))}^2\lesssim \mathcal{M}^3. $$
This and \eqref{L3smallannulus} implies that for $\mathcal{M}_{0}$ sufficiently large we have
\begin{align}\label{locenergyL3smalltruncation}
\begin{split}
&\|v_{\lambda_{(1)}}\|^{2}_{L^{\infty}_{t}L^{2}_{x}\cap L^{2}_{t}\dot{H}^{1}(B(y_0,4)\times (-\mathcal{M}^{-95},0))}+\|p_{\lambda_{(1)}}\|_{L^{\frac{3}{2}}_{x,t}(B(y_0,4)\times (-\mathcal{M}^{-95},0))}\leq \mathcal{M}^5\\
&\textrm{and}\quad \int\limits_{B(y_0,3)}|v_{\lambda_{(1)}}(x,-\mathcal{M}^{-95})|^3 dx\leq \epsilon_{*}.
\end{split} 
\end{align}
This allows us to apply Proposition \ref{KMT} and Remark \ref{KMTsmallertime} with $M$ replaced by $\mathcal{M}^5$. Thus, one deduces that for $\mathcal{M}_{0}$ sufficiently large
$$\|\nabla^{j} v_{\lambda_{(1)}}\|_{L^{\infty}_{x,t}(B(y_0,2)\times (-\mathcal{M}^{-96},0))}\leq \mathcal{M}^{50j+50}\quad\textrm{for}\,\, j=0,1,2. $$
Since this holds for any $y_0$ such that $B(y_0,4)\subset A_{k_0}$, we infer that
\begin{equation}\label{vlambda1higherderivatives}
\|\nabla^{j} v_{\lambda_{(1)}}\|_{L^{\infty}_{x,t}(\{4^{k_0}+2<|x|<4^{k_0+1}-2\}\times (-\mathcal{M}^{-96},0))}\leq \mathcal{M}^{50j+50}\quad\textrm{for}\,\, j=0,1,2.
\end{equation}
\textbf{Step 2: pressure and time derivative estimates inside the quantitative region of regularity}\\
Next, we want to use \eqref{plambda1large} and \eqref{vlambda1higherderivatives} to gain higher regularity of $p_{\lambda_{(1)}}$ and $\partial_{t}v_{\lambda_{(1)}}$ inside the region $\{x:4^{k_0}+3<|x|<4^{k_0+1}-3\}\times (-\mathcal{M}^{-96},0) $ .\\
Let the cut-off function $\varphi\in C^{\infty}(\mathbb{R}^3; [0,1])$ be such that
\begin{itemize}
\item[]$\varphi\equiv 1$ on $\{x:4^{k_0}+\tfrac{5}{2}<|x|<4^{k_0+1}-\tfrac{5}{2}\}$,
\medskip
\item[]$\supp\varphi\subset \{x:4^{k_0}+2<|x|<4^{k_0+1}-2\}$ and\medskip
\item [] for $j\in\mathbb{N}\quad \|\nabla^{j}\varphi\|_{L^{\infty}(\mathbb{R}^3)}\leq C_{j}$.  
\end{itemize}
Here, $C_{j}$ is a positive constant depending only on $j$. Let us also recall from the previous step that $A_{k_0}:=\{x:4^{k_0}<|x|\leq 4^{k_0+1}\}\subset B(0,e^{\mathcal{M}^5}).$ Thus,
\begin{equation}\label{measureannulus}
|A_{k_0} |\lesssim e^{3\mathcal{M}^5}.
\end{equation}
Here, $|A_{k_0}|$ denotes the Lebesgue measure of $A_{k_0}$.

For every $s\in (-\mathcal{M}^{-96},0)$ we decompose the pressure
$$p_{\lambda_{1}}(x,s)=p_{\textrm{Riesz},\lambda_{(1)}}(x,s)+p_{\textrm{h},\lambda_{(1)}}(x,s)\quad\textrm{with}\quad p_{\textrm{Riesz},\lambda_{(1)}}(x,s)=\mathcal{R}_{i}\mathcal{R}_{j}(\varphi v_{i}v_{j}(\cdot,s)). $$
Here, $R_{i}$ denote Riesz transforms and the Einstein summation convention has been adopted. Let us note that this decomposition implies that for every $s\in (-\mathcal{M}^{-96},0)$, $p_{\textrm{h},\lambda_{(1)}}(\cdot,s) $ is harmonic in $\{x:4^{k_0}+\tfrac{5}{2}<|x|<4^{k_0+1}-\tfrac{5}{2}\}. $ 

First let us estimate $p_{\textrm{Riesz},\lambda_{(1)}}$ and $\nabla p_{\textrm{Riesz}}$. By the Calder\'{o}n-Zygmund theory, H\"{o}lder's inequality and \eqref{vlambda1higherderivatives}-\eqref{measureannulus}, we infer the following for $\mathcal{M}_{0}$ sufficiently large. Namely, for any $q\in (1,\infty)$ and $s\in (-\mathcal{M}^{-96},0)$
\begin{align}\label{prieszfirst}
\begin{split}
\|p_{\textrm{Riesz},\lambda_{(1)}}(\cdot,s)\|_{L^{q}(\mathbb{R}^3)}\lesssim_{q}&\|v_{\lambda_{(1)}}(\cdot,s)\|^2_{L^{2q}(4^{k_0}+2<|x|<4^{k_0+1}-2)}\lesssim \|v_{\lambda_{(1)}}(\cdot,s)\|^2_{L^{\infty}(4^{k_0}+2<|x|<4^{k_0+1}-2)}|A_{k_0}|^{\frac{1}{q}}\\
\lesssim_{q}& \mathcal{M}^{100}e^{3\mathcal{M}^5}\leq e^{4\mathcal{M}^5}.
\end{split}
\end{align}
Similarly, for any $q\in (1,\infty)$ and $s\in (-\mathcal{M}^{-96},0)$ we obtain
\begin{align}\label{gradprieszfirst}
\begin{split}
\|\nabla p_{\textrm{Riesz},\lambda_{(1)}}(\cdot,s)\|_{L^{q}(\mathbb{R}^3)}\lesssim_{q}&\|v_{\lambda_{(1)}}(\cdot,s)\|^2_{L^{2q}(4^{k_0}+2<|x|<4^{k_0+1}-2)}+\||v_{\lambda_{(1)}}||\nabla v_{\lambda_{(1)}}|(\cdot,s)\|_{L^{q}(4^{k_0}+2<|x|<4^{k_0+1}-2)}\\
\lesssim_{q} &\|v_{\lambda_{(1)}}(\cdot,s)\|^2_{L^{\infty}(4^{k_0}+2<|x|<4^{k_0+1}-2)}|A_{k_0}|^{\frac{1}{q}}\\
&+\|v_{\lambda_{(1)}}(\cdot,s)\|_{L^{\infty}(4^{k_0}+2<|x|<4^{k_0+1}-2)}\|\nabla v_{\lambda_{(1)}}(\cdot,s)\|_{L^{\infty}(4^{k_0}+2<|x|<4^{k_0+1}-2)}|A_{k_0}|^{\frac{1}{q}}\\
\lesssim& \mathcal{M}^{150}e^{3\mathcal{M}^5}\leq e^{4\mathcal{M}^5}.
\end{split}
\end{align}
From \eqref{prieszfirst}-\eqref{gradprieszfirst} and the Sobolev embedding theorem, we have that 
\begin{equation}\label{Prieszbounded}
\| p_{\textrm{Riesz},\lambda_{(1)}}\|_{L^{\frac{3}{2}}_{t}L^{\infty}_{x}(\mathbb{R}^3\times (-\mathcal{M}^{-96},0))}\lesssim e^{4\mathcal{M}^5}.
\end{equation}
Furthermore, for all $q\in (1,\infty)$ we have
\begin{equation}\label{Prieszgradsecondest}
\| p_{\textrm{Riesz},\lambda_{(1)}}\|_{L^{\frac{3}{2}}_{t}L^{q}_{x}(\mathbb{R}^3\times (-\mathcal{M}^{-96},0))}+\| \nabla p_{\textrm{Riesz},\lambda_{(1)}}\|_{L^{\frac{3}{2}}_{t}L^{q}_{x}(\mathbb{R}^3\times (-\mathcal{M}^{-96},0))}\lesssim_{q} e^{4\mathcal{M}^5}.
\end{equation}

Next, we focus on estimating $p_{\textrm{h},\lambda_{(1)}}$ and $\nabla p_{\textrm{h},\lambda_{(1)}}$. Let us recall that for every $s\in (-\mathcal{M}^{-96},0)$, $p_{\textrm{h},\lambda_{(1)}}(\cdot,s) $ is harmonic in $\{x:4^{k_0}+\tfrac{5}{2}<|x|<4^{k_0+1}-\tfrac{5}{2}\}. $  Let $x_0\in\{x: 4^{k_0}+3<|x|<4^{k_0+1}-3\}$ and take $s\in (-\mathcal{M}^{-96},0)$. From properties of harmonic functions and the pressure decomposition, we infer that
\begin{align*}
|p_{\textrm{h},\lambda_{(1)}}(x_0,s)|+|\nabla p_{\textrm{h},\lambda_{(1)}}(x_0,s)|&\lesssim \|p_{\textrm{h},\lambda_{(1)}}(\cdot,s)\|_{L^{\frac{3}{2}}(B(x_0,\tfrac{1}{2}))}\\
&\lesssim \|p_{\textrm{Riesz},\lambda_{(1)}}(\cdot,s)\|_{L^{\frac{3}{2}}(\mathbb{R}^3)}+\|p_{\lambda_{(1)}}(\cdot,s)\|_{L^{\frac{3}{2}}(B(2e^{\mathcal{M}^5}))}.
\end{align*}
Thus for every $s\in (-\mathcal{M}^{-96},0)$.
\begin{align*}
\|p_{\textrm{h},\lambda_{(1)}}(\cdot,s)\|_{L^{\infty}(4^{k_0}+3<|x|<4^{k_0+1}-3)}+&\|\nabla p_{\textrm{h},\lambda_{(1)}}(\cdot,s)\|_{L^{\infty}(4^{k_0}+3<|x|<4^{k_0+1}-3)}\\
&\lesssim \|p_{\textrm{Riesz},\lambda_{(1)}}(\cdot,s)\|_{L^{\frac{3}{2}}(\mathbb{R}^3)}+\|p_{\lambda_{(1)}}(\cdot,s)\|_{L^{\frac{3}{2}}(B(2e^{\mathcal{M}^5}))}.
\end{align*}
$$ $$
Using this, \eqref{plambda1large} and \eqref{Prieszgradsecondest} we get that
\begin{equation}\label{pharmboundedest}
\|p_{\textrm{h},\lambda_{(1)}}\|_{L^{\frac{3}{2}}_{t}L^{\infty}_{x}(\{4^{k_0}+3<|x|<4^{k_0+1}-3\}\times (-\mathcal{M}^{-96},0))}+\|\nabla p_{\textrm{h},\lambda_{(1)}}\|_{L^{\frac{3}{2}}_{t}L^{\infty}_{x}(\{4^{k_0}+3<|x|<4^{k_0+1}-3\}\times (-\mathcal{M}^{-96},0))}\lesssim e^{4\mathcal{M}^5}.
\end{equation} 
This, H\"{o}lder's inequality and \eqref{measureannulus} imply that for all $q\in (1,\infty)$ 
\begin{equation}\label{pharmLqest}
\|p_{\textrm{h},\lambda_{(1)}}\|_{L^{\frac{3}{2}}_{t}L^{q}_{x}(\{4^{k_0}+3<|x|<4^{k_0+1}-3\}\times (-\mathcal{M}^{-96},0))}+\|\nabla p_{\textrm{h},\lambda_{(1)}}\|_{L^{\frac{3}{2}}_{t}L^{q}_{x}(\{4^{k_0}+3<|x|<4^{k_0+1}-3\}\times (-\mathcal{M}^{-96},0))}\lesssim_{q} e^{7\mathcal{M}^5}. 
\end{equation} 
Combining \eqref{Prieszbounded}-\eqref{pharmLqest} yields that 
\begin{equation}\label{pboundedestfinal}
\|p_{\lambda_{(1)}}\|_{L^{\frac{3}{2}}_{t}L^{\infty}_{x}(\{4^{k_0}+3<|x|<4^{k_0+1}-3\}\times (-\mathcal{M}^{-96},0))}\lesssim e^{4\mathcal{M}^5}
\end{equation} 
and that for all $q\in (1,\infty)$
\begin{equation}\label{pLqestfinal}
\|p_{\lambda_{(1)}}\|_{L^{\frac{3}{2}}_{t}L^{q}_{x}(\{4^{k_0}+3<|x|<4^{k_0+1}-3\}\times (-\mathcal{M}^{-96},0))}+\|\nabla p_{\lambda_{(1)}}\|_{L^{\frac{3}{2}}_{t}L^{q}_{x}(\{4^{k_0}+3<|x|<4^{k_0+1}-3\}\times (-\mathcal{M}^{-96},0))}\lesssim_{q} e^{7\mathcal{M}^5}. 
\end{equation} 

Finally, let us estimate $\partial_{t}v_{\lambda_{(1)}}$.
Using $(v_{\lambda_{(1)}},p_{\lambda_{(1)}})$ is a smooth solution to the Navier-Stokes equations on $B(0,4e^{\mathcal{M}^5})\times (-16e^{2\mathcal{M}^5},0]$, we see that the estimates \eqref{vlambda1higherderivatives} and \eqref{pLqestfinal} (together with H\"{o}lder's inequality) imply that for all $q\in (1,\infty)$
\begin{equation}\label{timederLqestfinal}
\|\partial_{t}v_{\lambda_{(1)}}\|_{L^{\frac{3}{2}}_{t}L^{q}_{x}(\{4^{k_0}+3<|x|<4^{k_0+1}-3\}\times (-\mathcal{M}^{-96},0))}\lesssim_{q} e^{7\mathcal{M}^5}. 
\end{equation}
\textbf{Step 3: Bogovskiĭ operator estimates}\\
Let $k_0\in \{2,\dots \lceil{\mathcal{M}^4}\rceil+1\}$ be as in the previous steps and define $R_{k_0}:=4^{k_0}+3.$ Note that 
\begin{align}\label{Rk0properties}
\begin{split}
&6\leq R_{k_0}\leq e^{\mathcal{M}^5},\quad A(R_{k_0}):=\{x: R_{k_0}<|x|<2R_{k_0}\}\subset \{x:4^{k_0}+3<|x|<4^{k_0+1}-3\}\quad\textrm{and}\\
&\quad |A(R_{k_0})|\lesssim e^{3\mathcal{M}^5}.
\end{split}
\end{align}
Let the cut-off function $\Phi\in C^{\infty}(\mathbb{R}^3; [0,1])$ be such that
\begin{itemize}
\item[]$\Phi\equiv 1$ on $B(0,R_{k_0})$,
\medskip
\item[]$\supp\Phi\subset B(0,2R_{k_0})$ and\medskip
\item [] for $j\in\mathbb{N} \qquad \|\nabla^{j}\Phi\|_{L^{\infty}(\mathbb{R}^3)}\lesssim_{j} R_{k_0}^{-j}$.
\end{itemize}
Then for every $t\in (-\mathcal{M}^{-96},0]$ we see that
\begin{equation}\label{vlambda1gradcutcompact}
v_{\lambda_{(1)}}(x,t)\cdot\nabla\Phi(x)\in C^{\infty}_{0}(A(R_{k_0})).
\end{equation}
Furthermore, using that $v_{\lambda_{(1)}}$ is divergence-free, we see that for every $t\in (-\mathcal{M}^{-96},0]$
\begin{align}\label{vlambda1gradcutzeroaverage}
\begin{split}
\int\limits_{A(R_{k_0})} v_{\lambda_{(1)}}(x,t)\cdot\nabla \Phi(x) dx&=\int\limits_{A(R_{k_0})} \textrm{div}\,(v_{\lambda_{(1)}}(x,t)\Phi(x)) dx=\int\limits_{\partial B(0,2R_{k_0})} v_{\lambda_{(1)}}\Phi\cdot dS-\int\limits_{\partial B(0,R_{k_0})} v_{\lambda_{(1)}}\Phi\cdot dS\\
&=-\int\limits_{\partial B(0,R_{k_0})} v_{\lambda_{(1)}}\cdot dS=-\int\limits_{B(0,R_{k_0})}\textrm{div}\, v_{\lambda_{(1)}} dx=0.
\end{split}
\end{align}
Using this, we can apply the Bogovskiĭ operator in Lemma \ref{lem:bogovskii} to $-v_{\lambda_{(1)}}\cdot\nabla \Phi $ on $A(R_{k_0})$ for each $t\in (-\mathcal{M}^{-96},0]$, which we denote by $w_{\lambda_{(1)}}:=B(-v_{\lambda_{(1)}}\cdot\nabla \Phi )$. Furthermore, by \eqref{vlambda1gradcutcompact} we have that for every $t\in (-\mathcal{M}^{-96},0]$
\begin{equation}\label{bogovcompactsupport}
w_{\lambda_{(1)}}(\cdot,t)\in C^{\infty}_{0}(A(R_{k_0})).
\end{equation}

Let us now estimate the spatial derivatives of $w_{\lambda_{(1)}}:=B(-v_{\lambda_{(1)}}\cdot\nabla \Phi )$. Using \eqref{Rk0properties} and \eqref{vlambda1higherderivatives} allows us to apply \eqref{timeindepbogest} of Lemma \ref{lem:bogovskii} to get
$$
\|\nabla w_{\lambda_{(1)}}\|_{L^{\infty}_{t}W^{2,4}_{x}(A(R_{k_0})\times (-\mathcal{M}^{-96},0))}\lesssim \|v_{\lambda_{(1)}}\cdot\nabla\Phi\|_{L^{\infty}_{t}W^{2,4}_{x}(A(R_{k_0})\times (-\mathcal{M}^{-96},0))}\leq e^{\mathcal{M}^5}.
$$
Using this, \eqref{Rk0properties}-\eqref{bogovcompactsupport} and Poincar\'{e}'s inequality gives that
\begin{equation}\label{bogov3derivatives}
\|w_{\lambda_{(1)}}\|_{L^{\infty}_{t}W^{3,4}_{x}(A(R_{k_0})\times (-\mathcal{M}^{-96},0))}\leq e^{3\mathcal{M}^5}.
\end{equation}
This and the Sobolev embedding theorem gives that
\begin{equation}\label{bogovboundeder}
\|w_{\lambda_{(1)}}\|_{L^{\infty}_{t}W^{2,\infty}_{x}(A(R_{k_0})\times (-\mathcal{M}^{-96},0))}\lesssim e^{3\mathcal{M}^5}.
\end{equation}
Thus, by H\"{o}lder's inequality we have
\begin{equation}\label{bogovder116}
\|w_{\lambda_{(1)}}\|_{L^{\frac{3}{2}}_{t}W^{2,\infty}_{x}(A(R_{k_0})\times (-\mathcal{M}^{-96},0))}\lesssim e^{3\mathcal{M}^5}.
\end{equation}
By \eqref{bogovboundeder}, \eqref{Rk0properties} and H\"{o}lder's inequality, we get
\begin{equation}\label{bogovL103}
\|w_{\lambda_{(1)}}\|_{L^{\infty}_{t}L^{\frac{10}{3}}_{x}(A(R_{k_0})\times (-\mathcal{M}^{-96},0))}\leq e^{4\mathcal{M}^5} 
\end{equation}
and
\begin{equation}\label{bogovL116}
\|w_{\lambda_{(1)}}\|_{L^{\frac{3}{2}}_{t}W^{2,\frac{11}{6}}_{x}(A(R_{k_0})\times (-\mathcal{M}^{-96},0))}\lesssim e^{\frac{51}{11}\mathcal{M}^5}.
\end{equation}

Let us now estimate the time derivative of the Bogovskiĭ operator $w_{\lambda_{(1)}}:=B(-v_{\lambda_{(1)}}\cdot\nabla \Phi )$. For every $t\in (-\mathcal{M}^{-96},0)$ we see that
$$
\partial_{t} v_{\lambda_{(1)}}(x,t)\cdot\nabla\Phi(x)\in C^{\infty}_{0}(A(R_{k_0})).
$$
Using this, Lemma \ref{lem:bogovskii} (specifically, \eqref{Bcommutewithpt}) and \eqref{timederLqestfinal}, we see that for every $t\in (-\mathcal{M}^{-96},0)$ 
\begin{equation}\label{bogovcompactsupporttimeder}
\partial_{t} w_{\lambda_{(1)}}(\cdot,t)\in C^{\infty}_{0}(A(R_{k_0}))
\end{equation}
and
\begin{equation}\label{bogovtimederest1}
\|\nabla \partial_{t} w_{\lambda_{(1)}}\|_{L^{\frac{3}{2}}_{t}L^{4}_{x}(A(R_{k_0})\times (-\mathcal{M}^{-96},0))}\lesssim \|\partial_{t} v_{\lambda_{(1)}}\cdot\nabla\Phi\|_{L^{\frac{3}{2}}_{t}L^{4}_{x}(A(R_{k_0})\times (-\mathcal{M}^{-96},0))}\lesssim e^{7\mathcal{M}^5}.
\end{equation}
Noting \eqref{Rk0properties} and \eqref{bogovcompactsupporttimeder}-\eqref{bogovtimederest1}, we apply Poincar\'{e}'s inequality to get
\begin{equation}\label{bogovtimedersobL4}
\| \partial_{t} w_{\lambda_{(1)}}\|_{L^{\frac{3}{2}}_{t}W^{1,4}_{x}(A(R_{k_0})\times (-\mathcal{M}^{-96},0))}\lesssim e^{8\mathcal{M}^5}.
\end{equation}
Hence, the Sobolev embedding theorem gives
\begin{equation}\label{bogovtimederbounded}
\|\partial_{t} w_{\lambda_{(1)}}\|_{L^{\frac{3}{2}}_{t}L^{\infty}_{x}(A(R_{k_0})\times (-\mathcal{M}^{-96},0))}\lesssim e^{8\mathcal{M}^5}.
\end{equation}
Observing \eqref{Rk0properties} and \eqref{bogovtimederbounded}, we can apply H\"{o}lder's inequality to conclude that
\begin{equation}\label{bogovtimederL116}
\| \partial_{t} w_{\lambda_{(1)}}\|_{L^{\frac{3}{2}}_{t}L^{\frac{11}{6}}_{x}(A(R_{k_0})\times (-\mathcal{M}^{-96},0))}\lesssim e^{\frac{106}{11}\mathcal{M}^5}.
\end{equation}
\textbf{Step 4: Truncating the solution and rescaling to conclude}\\
Let $\Phi$ and $w_{\lambda_{(1)}}$ be as in the previous step. Let us now define the truncated solution
\begin{equation}\label{truncatedsollambda1def}
u_{\lambda_{(1)}}(x,t):=\Phi v_{\lambda_{(1)}}(x,t)+w_{\lambda_{(1)}}(x,t)\quad\textrm{and}\quad \pi_{\lambda_{(1)}}(x,t):=\Phi p_{\lambda_{(1)}}(x,t)\quad\textrm{for}\quad (x,t)\in \mathbb{R}^3\times (-\mathcal{M}^{-96},0].
\end{equation}
From Lemma \ref{lem:bogovskii}, properties of the cut-off function $\Phi$ and \eqref{bogovcompactsupport}, we deduce the following. Namely $u_{\lambda_{(1)}}$ is divergence-free and that for every $t\in (-\mathcal{M}^{-96},0]$,  $u_{\lambda_{(1)}}$ and $\pi_{\lambda_{(1)}}$ are smooth with compact support in $B(0, 2R_{k_0})$. Furthermore, we see that $(u_{\lambda_{(1)}}, \pi_{\lambda_{(1)}})$ are smooth on $\mathbb{R}^3\times(-\mathcal{M}^{-96},0)$ and $B(0, R_{k_0})\times [-\mathcal{M}^{-96},0] $. Furthermore, $(u_{\lambda_{(1)}}, \pi_{\lambda_{(1)}})$ solve the forced Navier-Stokes equations in $\mathbb{R}^3\times (-\mathcal{M}^{-96},0)$ with forcing term given by
\begin{equation}\label{forcinglambda1}
\begin{gathered}
	f_{\lambda_{(1)}} :=  - \Delta \Phi v_{\lambda_{(1)}} - 2 \nabla \Phi \cdot \nabla v_{\lambda_{(1)}} + \Phi v_{\lambda_{(1)}} \cdot \nabla \Phi v_{\lambda_{(1)}} + (\Phi^2 - \Phi) v_{\lambda_{(1)}} \cdot \nabla v_{\lambda_{(1)}} \\ + (\p_t - \Delta) w_{\lambda_{(1)}} + \Phi v_{\lambda_{(1)}} \cdot \nabla w_{\lambda_{(1)}} + w_{\lambda_{(1)}} \cdot \nabla (\Phi v_{\lambda_{(1)}}) + w_{\lambda_{(1)}} \cdot \nabla w_{\lambda_{(1)}} + \nabla \Phi p_{\lambda_{(1)}}.
	\end{gathered}
	\end{equation}
	Using properties of the cut-off function $\Phi$, \eqref{vlambda1higherderivatives}, \eqref{pboundedestfinal}-\eqref{pLqestfinal}, \eqref{bogovboundeder}-\eqref{bogovder116}, \eqref{bogovL116}
and \eqref{bogovtimederbounded}-\eqref{bogovtimederL116}, we see that
\begin{equation}\label{flambda1norm}
\|f_{\lambda_{(1)}}\|_{L^{\frac{3}{2}}_{t}L^{\infty}_{x}(A(R_{k_0})\times (-\mathcal{M}^{-96},0))}+\|f_{\lambda_{(1)}}\|_{L^{\frac{3}{2}}_{t}L^{\frac{11}{6}}_{x}(A(R_{k_0})\times (-\mathcal{M}^{-96},0))}\leq e^{12\mathcal{M}^5}.
\end{equation}
From properties of the cut-off function $\Phi$ and \eqref{bogovcompactsupport}, we see that for every $t\in (-\mathcal{M}^{-96},0)$
\begin{equation}\label{flambda1support}
\supp f_{\lambda_{(1)}}(\cdot,t)\subseteq B(0,2R_{k_0})\setminus \overline{B(0,R_{k_0})}\quad\textrm{and}\quad \supp u_{\lambda_{(1)}}(\cdot,t)\subseteq B(0, 2R_{k_0})\quad\textrm{with}\quad R_{k_0}\geq 6.
\end{equation}
From \eqref{vlambda1L3}, properties of the cut-off function $\Phi$ and \eqref{bogovL103} one deduces that $u_{\lambda_{(1)}}=\Phi v_{\lambda_{(1)}}+w_{\lambda_{(1)}}$ with
\begin{align}\label{vlambda1decomp}
&\|\Phi v_{\lambda_{(1)}}\|_{L^{\infty}_{t}L^{3}_{x}(\mathbb{R}^3\times (-\mathcal{M}^{-96},0))}=\| u_{\lambda_{(1)}}\|_{L^{\infty}_{t}L^{3}_{x}(B(0,R_{k_0})\times (-\mathcal{M}^{-96},0))}\leq\mathcal{M}\quad\textrm{and}\quad\\
& \|w_{\lambda_{(1)}}\|_{L^{\infty}_{t}L^{\frac{10}{3}}_{x}(\mathbb{R}^3\times (-\mathcal{M}^{-96},0))}\leq e^{4{\mathcal{M}^5}}.
\end{align} 
From properties of the cut-off function $\Phi$ and \eqref{bogovcompactsupport}, we see that
\begin{align}\label{ulambda1simvlambda1}
\begin{split}
(u_{\lambda_{(1)}}(x,t), \pi_{\lambda_{(1)}}(x,t))&=(v_{\lambda_{(1)}}(x,t), p_{\lambda_{(1)}}(x,t))=(\lambda_{(1)}v(\lambda_{(1)}x, \lambda_{(1)}^2 t), \lambda_{(1)}^2 p(\lambda_{(1)}x,\lambda_{(1)}^2 t))\\&\textrm{on}\,\,B(0,R_{k_0})\times [-\mathcal{M}^{-96},0]\supseteq B(0,6)\times [-\mathcal{M}^{-96},0].
\end{split}
\end{align} 
Using this, together with \eqref{Morreyboundtruncationrescale} and \eqref{Rk0properties}, gives that
\begin{equation}\label{vlambda_1morrey}
\sup_{ 0<r\leq \mathcal{M}^{-48}}\left\{\frac{1}{r^2}\int\limits_{Q(0,r)} |\pi_{\lambda_1}|^{\frac{3}{2}}dxdt
+\frac{1}{r}\|u_{\lambda_{(1)}}\|^2_{L^{\infty}_{t}L^{2}_{x}\cap L^{2}\dot{H}^{1}(Q(z_0,r))}\right\}
\lesssim \mathcal{M}^3.
\end{equation}
We now define the functions $V:\mathbb{R}^3\times [-2,0]\rightarrow\mathbb{R}^3$, $F:\mathbb{R}^3\times [-2,0]\rightarrow\mathbb{R}^3$ and $P:\mathbb{R}^3\times [-2,0]\rightarrow\mathbb{R}$ by
\begin{align}\label{VPFdef}
\begin{split}
(V(x,t), P(x,t), F(x,t)):&=(\lambda_{(2)}u_{\lambda_{(1)}}(\lambda_{(2)}x,\lambda_{(2)}^2 t), \lambda_{(2)}^2\pi_{\lambda_{(1)}}(\lambda_{(2)}x,\lambda_{(2)}^2 t),\lambda_{(2)}^3f_{\lambda_{(1)}}(\lambda_{(2)}x,\lambda_{(2)}^2 t))\\&\textrm{with}\quad \lambda_{(2)}:=e^{-e^{2\mathcal{M}}}.
\end{split}
\end{align}
We also define
\begin{equation}\label{Rdef}
R:=R_{k_0}\lambda_{(2)}^{-1}.
\end{equation}
With these definitions it is easy to see that $V$ and $P$ are smooth on $\mathbb{R}^3\times [-2,0)$. Additionally by \eqref{ulambda1simvlambda1}, we see that $V$ and $P$ are smooth on $B(0,R)\times [-2,0]$. Furthermore, the properties \eqref{flambda1norm}-\eqref{vlambda_1morrey} for $(u_{\lambda_{(1)}},\pi_{\lambda_{(1)}}, f_{\lambda_{(1)}})$ and $R_{k_0}$ imply the desired properties for $(V,P,F)$ and $R$ stated in Proposition \ref{truncationprocedure}.  
\end{proof}
\section{Quantitative Regions of Regularity for the Truncated Solution}

\subsection{Quantitative Epochs of Regularity for the Truncated Solution}
\begin{pro}\label{epoch}
There exists $\mathcal{M}_{0}$ sufficiently large such that the following holds true. Let $(v,p)$, $(V,P)$, $F$ and $R$ be as in the statement of Proposition \ref{truncationprocedure}. Let $0<T_1\leq 1$ and define
$$ I:=(-T_1, -\tfrac{T_1}{2})\subset[-1,0]\subset[-2,0].$$
Then the above assumptions imply that there exists an epoch of regularity $I'\subset I$, which is a closed interval such that
\begin{equation}\label{epochlength}
|I'|=\mathcal{M}^{-46}|I|\,\,\,\textrm{and}
\end{equation}
\begin{equation}\label{epochvbounds}
\sup_{(x,t)\in B(0,\tfrac{R}{2})\times I'}|\nabla ^j V(x,t)|\lesssim \frac{1}{\mathcal{M}} |I'|^{-\frac{j+1}{2}}\quad\textrm{for}\,\,j=0,1,2.
\end{equation}
\end{pro}
\begin{proof}

\textbf{Step 1: Preliminary estimates}\\
The properties of $(V,P,F)$ and $R$ (as stated in Proposition \ref{truncationprocedure}) are also satisfied for the rescaled quantities
$$(V_{T_1}(x,t),P_{T_1}(x,t), F_{T_1}(x,t)):=(T_1^{\frac{1}{2}}V(T_1^{\frac{1}{2}}x,T_1 t),T_1P(T_1^{\frac{1}{2}}x, T_1t),T_1^{\frac{3}{2}}F(T_1^{\frac{1}{2}}x, T_1t))\quad\textrm{and}\quad R_{T_1}:=\frac{R}{T_1^{\frac{1}{2}}}$$ with any $0<T_1\leq 1.$ Hence, without loss of generality we can take $T_1=1$ ($I:=(-1,-\tfrac{1}{2})$) in Proposition \ref{epoch}.

On $\mathbb{R}^3\times (-2,0)$ we have
\begin{equation}\label{Vmild}
V(\cdot,t)=L(F)(\cdot,t)+B(V,V)(\cdot,t).
\end{equation}
Here, \begin{equation}\label{LFdef}
L(F)(\cdot,t):=e^{(t+2)\Delta}V(\cdot,-2)+\int\limits_{-2}^{t} e^{(t-s)\Delta}\mathbb{P}F(\cdot,s) ds\,\,\textrm{and}\,\,\,
B(A,B)(\cdot,t):=-\int\limits_{-2}^{t} e^{(t-s)\Delta}\mathbb{P}\nabla\cdot(A\otimes B)(\cdot,s) ds.
\end{equation}
Here, $e^{t\Delta}$ denotes the heat semigroup on $\mathbb{R}^3$ and $\mathbb{P}$ denotes the projection onto divergence-free vector fields in $\mathbb{R}^3$.
Using $V=V^{(1)}+V^{(2)}$ on $\mathbb{R}^3\times [-2,0]$, \eqref{Vdecompbound} and standard estimates of the heat semigroup yields 
 \begin{equation}\label{heat-2}
 \sup_{t\in [-1,0]}\|e^{t\Delta}V(\cdot,-2)\|_{L^{q}(\mathbb{R}^3)}\lesssim_{q} \mathcal{M}\quad\forall q\in [\tfrac{10}{3},\infty].
 \end{equation}
 Using standard estimates of the heat semigroup and Leray projector, together with H\"{o}lder's inequality,
 yields that for $t\in [-2,0)$
 \begin{align}\label{integFest}
 \begin{split}
  &\Big\|\int\limits_{-2}^{t} e^{(t-s)\Delta}\mathbb{P}F(\cdot,s) ds\Big\|_{L^{\frac{11}{6}}(\mathbb{R}^3)}\lesssim \|F\|_{L^{\frac{3}{2}}_{t}L^{\frac{11}{6}}_{x}(\mathbb{R}^3\times (-2,0))}\,\,\textrm{and}\\
  &\Big\|\int\limits_{-2}^{t} e^{(t-s)\Delta}\mathbb{P}F(\cdot,s) ds\Big\|_{L^{\infty}(\mathbb{R}^3)}\lesssim\int\limits_{-2}^{t}\frac{\|F(\cdot,s)\|_{L^{5}(\mathbb{R}^3)}}{(t-s)^{\frac{3}{10}}} ds\lesssim \|F\|_{L^{\frac{3}{2}}_{t}L^{5}_{x}(\mathbb{R}^3\times (-2,0))}.
  \end{split} 
 \end{align}
 Using \eqref{heat-2}-\eqref{integFest} and \eqref{Festimate} gives that for $\mathcal{M}_{0}$ sufficiently large
\begin{equation}\label{LFest}
\sup_{t\in [-1,0)}\|L(F)(\cdot,t)\|_{L^{q}(\mathbb{R}^3)}\lesssim_{q} \mathcal{M}\quad\forall q\in [\tfrac{10}{3},\infty].
\end{equation}
Let us now estimate the kinetic energy of $B(V,V)(\cdot,t)$, which we will now denote by 
\begin{equation}\label{Wdef}
W(\cdot,t):=B(V^{(1)},V^{(1)})(\cdot,t)+B(V^{(1)},V^{(2)})(\cdot,t)+B(V^{(2)},V^{(1)})(\cdot,t)+B(V^{(2)},V^{(2)})(\cdot,t).
\end{equation}
 for convenience.
From \cite{oseen}, it is known that $e^{(t-s)\Delta}\mathbb{P}\nabla\cdot$ has an associated convolution kernel $K_{ijk}$ $(i,j,k\in \{1,2,3\})$.
Furthermore, from \cite{solonnikov} we have
\begin{equation}\label{solonnikovkernelest}
| K(x,t)|\lesssim\frac{1}{(|x|^2+t)^{2}}\quad\textrm{for}\quad (x,t)\in\mathbb{R}^3\times (0,\infty).
\end{equation}
Using \eqref{solonnikovkernelest}, \eqref{Vdecompbound} and Young's convolution inequality gives that for $t\in (-2,0)$
\begin{equation}\label{BV1V1est}
\|B(V^{(1)},V^{(1)})(\cdot,t)\|_{L^{2}(\mathbb{R}^3)}\lesssim \int\limits_{-2}^{t}\frac{1}{(t-s)^{\frac{3}{4}}}\|V^{(1)}(\cdot,s)\|_{L^{3}(\mathbb{R}^3)}^2 ds\lesssim \mathcal{M}^2,
\end{equation}
\begin{equation}\label{BV1V2est}
\|B(V^{(1)},V^{(2)})(\cdot,t)\|_{L^{2}(\mathbb{R}^3)}\lesssim \int\limits_{-2}^{t}\frac{1}{(t-s)^{\frac{7}{10}}}\|V^{(1)}(\cdot,s)\|_{L^{3}(\mathbb{R}^3)}\|V^{(2)}(\cdot,s)\|_{L^{\frac{10}{3}}(\mathbb{R}^3)} ds\lesssim \mathcal{M}e^{-e^{\mathcal{M}}}\quad\textrm{and}
\end{equation} 
\begin{equation}\label{BV2V2est}
\|B(V^{(2)},V^{(2)})(\cdot,t)\|_{L^{2}(\mathbb{R}^3)}\lesssim \int\limits_{-2}^{t}\frac{1}{(t-s)^{\frac{13}{20}}}\|V^{(2)}(\cdot,s)\|_{L^{\frac{10}{3}}(\mathbb{R}^3)}^2 ds\lesssim e^{-2e^{\mathcal{M}}}.
\end{equation}
Combining \eqref{BV1V1est}-\eqref{BV2V2est} yields that
\begin{equation}\label{Wkineticest}
\sup_{t\in (-2,0)}\|W(\cdot,t)\|_{L^{2}(\mathbb{R}^3)}\lesssim \mathcal{M}^2.
\end{equation}
\textbf{Step 2: Energy estimates for $W$ and higher integrability of $V$}\\
Observe that the smooth function $W$ satisfies the following equation on $\mathbb{R}^3\times (-2,0)$:
\begin{align}\label{Weqn}
\begin{split}
&\partial_{t}W-\Delta W+W\cdot\nabla W+W\cdot\nabla L(F)+L(F)\cdot\nabla W+L(F)\cdot\nabla L(F)+\nabla \Pi=0,\\
&\textrm{div}\,W=0\,\,\textrm{and}\,\,W(\cdot,-2)=0.
\end{split}
\end{align}
Performing a standard energy estimate on \eqref{Weqn} yields that for $t\in [-1,0)$
\begin{equation}\label{Wenergyident}
\|W(\cdot,t)\|_{L^{2}(\mathbb{R}^3)}^2+2\int\limits_{-1}^{t}\int\limits_{\mathbb{R}^3}|\nabla W(x,s)|^2 ds=\|W(\cdot,-1)\|_{L^{2}(\mathbb{R}^3)}^2+2\int\limits_{-1}^{t}\int\limits_{\mathbb{R}^3}(W+L(F))\otimes L(F):\nabla W dxds.
\end{equation}
Using this, H\"{o}lder's inequality and Young's algebraic inequality yields
that for $t\in [-1,0)$
\begin{align}\label{Wenergyest}
\|W(\cdot,t)\|_{L^{2}(\mathbb{R}^3)}^2+\int\limits_{-1}^{t}\int\limits_{\mathbb{R}^3}|\nabla W(x,s)|^2 ds\lesssim\|W(\cdot,-1)\|_{L^{2}(\mathbb{R}^3)}^2+\int\limits_{-1}^{t}\int\limits_{\mathbb{R}^3}|W|^2|L(F)|^2+|L(F)|^4 dxds.
\end{align}
By \eqref{LFest} and \eqref{Wkineticest}, we deduce from \eqref{Wenergyest} that
\begin{equation}\label{Wenergyestmain}
\sup_{t\in [-1,0)}\|W(\cdot,t)\|_{L^{2}(\mathbb{R}^3)}^2+\int\limits_{-1}^{0}\int\limits_{\mathbb{R}^3}|\nabla W(x,s)|^2 ds\lesssim \mathcal{M}^6.
\end{equation}
Using this, the Sobolev embedding theorem and the pigeonhole principle, we see that there exists $t_1\in (-1,-\tfrac{3}{4})$ such that
$$\|W(\cdot,t_1)\|_{L^{6}(\mathbb{R}^3)}\lesssim \mathcal{M}^3. $$
Combining this with \eqref{LFest} gives that
$$\|V(\cdot,t_1)\|_{L^{6}(\mathbb{R}^3)}\lesssim \mathcal{M}^3. $$
Using this estimate for $V(\cdot,t_1)$ and \eqref{Festimate}, one deduces that
\begin{equation}\label{LFt_1est}
\sup_{t\in [t_1,0)}\Big\|e^{(t-t_1)\Delta}V(\cdot, t_1)+\int\limits_{t_1}^{t} e^{(t-s)\Delta}\mathbb{P}F(\cdot,s) ds\Big\|_{L^{6}(\mathbb{R}^3)}\lesssim \mathcal{M}^3.
\end{equation}
Using this and applying \cite[Theorem (3.2)]{FJR}, we see that for $\mathcal{M}$ sufficiently large there exists a solution $\bar{V}:\mathbb{R}^3\times [t_1, t_1+{C_{univ}}{\mathcal{M}^{-12}}]\rightarrow\mathbb{R}^3$ to 
\begin{equation}\label{barVeqn}
\bar{V}(\cdot,t)=e^{(t-t_1)\Delta}V(\cdot, t_1)+\int\limits_{t_1}^{t} e^{(t-s)\Delta}\mathbb{P}F(\cdot,s) ds-\int\limits_{t_1}^{t} e^{(t-s)\Delta}\mathbb{P}\nabla\cdot(\bar{V}\otimes \bar{V})(\cdot,s) ds.
\end{equation}
with
\begin{equation}\label{barVest}
\sup_{s\in [t_1, t_1+{C_{univ}}{\mathcal{M}^{-12}}]}\|\bar{V}(\cdot, s)\|_{L^{6}(\mathbb{R}^3)}\lesssim \mathcal{M}^3.
\end{equation}
Here, $C_{univ}$ is a positive universal constant.
Since $V$ is smooth on $\mathbb{R}^3\times [-t_1,0)$ and with compact support \eqref{Vcompactsupport}, we have that for $\mathcal{M}$ sufficiently large that $$V\in L^{\infty}_{t}L^{6}_{x}(\mathbb{R}^3\times (t_1, t_1+{C_{univ}}{\mathcal{M}^{-12}})).$$
Furthermore, $V$ also satisfies the same equation \eqref{barVeqn} as $\bar{V}$.
So by \cite[Theorem (3.3)]{FJR}, we infer that $\bar{V}\equiv V $ on $\mathbb{R}^3\times [t_1, t_1+{C_{univ}}{\mathcal{M}^{-12}}]$. Hence, from \eqref{barVest} we obtain
\begin{equation}\label{VestLinfinityL6}
\sup_{s\in [t_1, t_1+{C_{univ}}{\mathcal{M}^{-12}}]}\|{V}(\cdot, s)\|_{L^{6}(\mathbb{R}^3)}\lesssim \mathcal{M}^3.
\end{equation} 
\textbf{Step 3: Higher differentiability for $V$ via $\varepsilon$-regularity }\\
The pressure $P$ can be decomposed as
\begin{equation}\label{Pdecomp}
P=P_{V\otimes V}+P_{F}\qquad\textrm{with}\quad P_{V\otimes V}:=\mathcal{R}_{i}\mathcal{R}_{j}(V_i V_j)\quad\textrm{and}\quad P_{F}:=(\Delta)^{-1}(\textrm{div}\,F).
\end{equation}
Using \eqref{Festimate}, \eqref{VestLinfinityL6} and Calder\'{o}n-Zygmund bounds gives that
\begin{equation}\label{Pdecompestimates}
\|P_{V\otimes V}\|_{L^{\infty}_{t}L^{3}_{x}(\mathbb{R}^3\times (t_1, t_1+{C_{univ}}{\mathcal{M}^{-12}}))}\lesssim \mathcal{M}^6\qquad\textrm{and}\qquad \|\nabla P_{F}\|_{L^{\frac{3}{2}}_{t}L^{\frac{11}{6}}_{x}(\mathbb{R}^3\times (t_1, t_1+{C_{univ}}{\mathcal{M}^{-12}}) )}\lesssim e^{-e^{\mathcal{M}}}.
\end{equation}
Next, we fix $x\in B(0,\frac{R}{2})$, $r:=\mathcal{M}^{-\frac{13}{2}}$ and $t\in [t_1+\tfrac{3}{2}\mathcal{M}^{-13}, t_1+2\mathcal{M}^{-13}].$ Using the properties of $R$ in Proposition \ref{truncationprocedure} and \eqref{VP=vprescale}, we see that
\begin{align}\label{VPpropep}
\begin{split}
&B(x,r)\times (t-r^2, t)\subseteq B(0,R)\times (t_1, t_1+{C_{univ}}{\mathcal{M}^{-12}}),\\
&(V,P)\,\,\textrm{solve}\,\,\textrm{the}\,\,\textrm{Navier-Stokes}\,\,\textrm{equations}\,\,\textrm{in}\,\, B(x,r)\times (t-r^2, t).
\end{split}
\end{align}
Using \eqref{VestLinfinityL6}, \eqref{VPpropep} and H\"{o}lder's inequality, one infers that
\begin{equation}\label{VCKN}
\frac{1}{r^2}\int\limits_{t-r^2}^{t}\int\limits_{B(x,r)} |V(y,s)|^3dyds \lesssim r^{\frac{3}{2}}\sup_{s\in [t_1, t_1+{C_{univ}}{\mathcal{M}^{-12}}]}\|{V}(\cdot, s)\|_{L^{6}(\mathbb{R}^3)}^3\lesssim \mathcal{M}^{-\frac{3}{4}}.
\end{equation}
The same arguments applied to $P_{V\otimes V}$ yield
\begin{equation}\label{PVoVCKN}
\frac{1}{r^2}\int\limits_{t-r^2}^{t}\int\limits_{B(x,r)} |P_{V\otimes V}(y,s)-(P_{V\otimes V})_{B(x,r)}(s)|^{\frac{3}{2}}dyds \lesssim  \mathcal{M}^{-\frac{3}{4}}.
\end{equation}
Next, using \eqref{Pdecompestimates}, \eqref{VPpropep}, Poincar\'{e}'s inequality and H\"{o}lder's inequality gives
\begin{align}\label{PFCKN}
\begin{split}
&\frac{1}{r^2}\int\limits_{t-r^2}^{t}\int\limits_{B(x,r)} |P_{F}(y,s)-(P_F)_{B(x,r)}(s)|^{\frac{3}{2}}dyds \lesssim \frac{1}{r^{\frac{1}{2}}}\int\limits_{t-r^2}^{t}\int\limits_{B(x,r)} |\nabla P_{F}|^{\frac{3}{2}}dyds\\
&\lesssim r^{\frac{1}{22}}\int\limits_{t-r^2}^{t}\Big(\int\limits_{B(x,r)} |\nabla P_{F}|^{\frac{11}{6}}dy\Big)^{\frac{9}{11}} ds\lesssim \mathcal{M}^{-\frac{13}{44}}e^{-\tfrac{3}{2}e^{\mathcal{M}}}.
\end{split}
\end{align}
Combining \eqref{VCKN}-\eqref{PFCKN} gives
\begin{equation}\label{VPCKN}
\frac{1}{r^2}\int\limits_{t-r^2}^{t}\int\limits_{B(x,r)}|V(y,s)|^3+|P(y,s)-(P)_{B(x,r)}(s)|^{\frac{3}{2}}dyds\lesssim \mathcal{M}^{-\frac{3}{4}}+\mathcal{M}^{-\frac{13}{44}}e^{-\tfrac{3}{2}e^{\mathcal{M}}}.
\end{equation}
Using \eqref{VPpropep} and Caffarelli, Kohn and Nirenberg's $\ep$-regularity criteria \cite{CKN} and \cite[Lemma 6.1]{gregory2014lecture}, we infer from \eqref{VPCKN} that for $\mathcal{M}$ sufficiently large
$$
|\nabla^{j}V(x,t)|\lesssim_{j} r^{-(j+1)}\lesssim \mathcal{M}^{\frac{39}{2}}\quad\textrm{for}\quad j=0,1,2.
$$
Thus, we have
\begin{equation}\label{VPepoch}
\sup_{B(0,\frac{R}{2})\times [t_1+\tfrac{3}{2}\mathcal{M}^{-13}, t_1+2\mathcal{M}^{-13}]} |\nabla^j V|\lesssim \mathcal{M}^{\frac{39}{2}}\quad\textrm{for}\quad j=0,1,2.
\end{equation}
Let \begin{equation}\label{epochdef}
I':=[t_1+\tfrac{3}{2}\mathcal{M}^{-13}, t_1+\tfrac{3}{2}\mathcal{M}^{-13}+\tfrac{1}{2}\mathcal{M}^{-46}]\subset [t_1+\tfrac{3}{2}\mathcal{M}^{-13}, t_1+2\mathcal{M}^{-13}].
\end{equation}
Then $|I'|=\tfrac{1}{2}\mathcal{M}^{-46}=|I|\mathcal{M}^{-46}$. From  this and \eqref{VPepoch}, we get
$$\sup_{B(0,\frac{R}{2})\times I'} |\nabla^j V|\lesssim \mathcal{M}^{\frac{39}{2}}\lesssim\frac{1}{M}|I'|^{-\frac{j+1}{2}}\quad\textrm{for}\quad j=0,1,2$$
as required.
\end{proof}
\subsection{Quantitative Annuli of Regularity for the Truncated Solution}
In order to show the existence of quantitative annuli of regularity for the truncated solution $V$, we will make use of the following proposition taken from \cite{barkerprange2020}.
\begin{pro}\label{CKNquanthigher}\cite[Proposition 6.4]{barkerprange2020}\\
There exists a universal constant $\ep_{1}^{*}\in (0,1)$, such that if $(v,p)$ is a suitable weak solution to the Navier-Stokes equations on $Q(0,1)$ and for some $\ep_{1}\leq \ep_{1}^{*}$ 
\begin{equation}\label{CKNsmallnessrepeat}
\int\limits_{Q(0,1)} |v|^3+|p|^{\frac{3}{2}} dxdt\leq \ep_{1}
\end{equation}
then  one concludes that for $j=0,1$
\begin{equation}\label{CKNboundedquanthigher}
\nabla^{j}v\in L^{\infty}(Q(0,\tfrac{1}{4}))\,\,\,\,\textrm{with}\,\,\,\,\|\nabla^j v\|_{L^{\infty}(Q(0,\frac{1}{4}))}\lesssim  \ep_{1}^{\frac{1}{3}}
\end{equation} 
and 
\begin{equation}\label{CKNboundedvort}
\nabla\omega\in L^{\infty}(Q(0,\tfrac{1}{4}))\,\,\,\,\textrm{with}\,\,\,\,\|\nabla \omega\|_{L^{\infty}(Q(0,\frac{1}{4}))}\lesssim  \ep_{1}^{\frac{1}{3}}.
\end{equation} 
\end{pro}
\begin{pro}\label{annulus}
Fix  any $\mu\geq 1$ and let $\mathcal{M}$, $(v,p)$, $(V,P)$, $F$ and $R$ be as in Proposition \ref{truncationprocedure}. Let $\Omega:=\nabla\times V$. There exists a positive constant $\lambda_{0}(\mu)$ such that if $\mathcal{M}\geq \lambda_{0}(\mu)$ then the following holds true.\\
In particular, for any $T_{1}\in (0,1]$ and 
\begin{equation}\label{R0condition}
T^{\frac{1}{2}}_1\leq R_{0}\leq  {T}^{\frac{1}{2}}_1 6e^{e^{2\mathcal{M}}}e^{-4\mu \mathcal{M}^{11\mu+22}}
\end{equation}
there exists $R_{2}$ such that
\begin{equation}\label{goodscaleR2}
2R_{0}\leq R_{2}\leq 2R_{0}e^{\mu \mathcal{M}^{11\mu+22}}\leq R
\end{equation}
such that
\begin{equation}\label{goodannulusquantest}
\|\nabla^{j} V\|_{L^{\infty}_{x,t}(R_{2}< |x|< \frac{R_{2}\mathcal{M}^{11\mu}}{4}\times (-\frac{T_1}{16},0))}\lesssim\mathcal{M}^{-3\mu}T_{1}^{-\frac{j+1}{2}}\quad\textrm{for}\,\,j=0,1
\end{equation}
and
\begin{equation}\label{gradvortbounded}
\|\nabla \Omega\|_{L^{\infty}_{x,t}(R_{2}< |x|< \frac{R_{2}\mathcal{M}^{11\mu}}{4}\times (-\frac{T_1}{16},0))}\lesssim\mathcal{M}^{-3\mu}T_{1}^{-\frac{3}{2}}.
\end{equation}
\end{pro}
\begin{proof}
\textbf{Step 1: Preliminary estimates}\\
From now on we take $\mathcal{M}\geq\mathcal{M}(\mu)$ such that
$$1\leq 6e^{e^{2\mathcal{M}}}e^{-4\mu \mathcal{M}^{11\mu+22}}\quad\forall\mathcal{M}\geq \mathcal{M}(\mu) .$$For identical reasons as described in the proof of Proposition \ref{epoch}, we may assume without loss of generality that $T_1=1$ by a scaling argument. As in the proof of Proposition \ref{epoch}, we decompose $V:\mathbb{R}^3\times [-2,0]\rightarrow\mathbb{R}^3$ as
$$V(\cdot,t)=W(\cdot,t)+L(F)(\cdot,t)\quad\textrm{for}\quad t\in (-2,0).$$ Here, $L(F)$ is defined in \eqref{LFdef} and $W$ is defined in \eqref{Wdef}.
 We decompose the pressure $P:\mathbb{R}^3\times [-2,0]\rightarrow\mathbb{R}$ as
\begin{equation}\label{pressuredecomp}
P(\cdot,t)=P_{V\otimes V}(\cdot,t)+P_{F}(\cdot,t)\quad\textrm{for}\quad t\in (-2,0).
\end{equation}
Here, $P_{V\otimes V}$ and $P_{F}$ are as in \eqref{Pdecomp}.

Using \eqref{LFest}, \eqref{Wenergyestmain} and the Sobolev embedding theorem gives
\begin{equation}\label{V103est}
\int\limits_{-1}^{0}\int\limits_{\mathbb{R}^3}|V(y,s)|^{\frac{10}{3}} dyds\lesssim \mathcal{M}^{10}.
\end{equation}
Using this and Calder\'{o}n-Zygmund estimates, we have
\begin{equation}\label{PVotimesVest}
\int\limits_{-1}^{0}\int\limits_{\mathbb{R}^3}|P_{V\otimes V}(y,s)|^{\frac{5}{3}} dyds\lesssim \mathcal{M}^{10}.
\end{equation}
Thus, for $\mathcal{M}$ sufficiently large we have
\begin{equation}\label{VPVotimesVmainest}
\int\limits_{-1}^{0}\int\limits_{\mathbb{R}^3}|V(y,s)|^{\frac{10}{3}}+|P_{V\otimes V}(y,s)|^{\frac{5}{3}} dyds\leq \lambda:=\mathcal{M}^{11}.
\end{equation}
Using \eqref{Festimate} and Calder\'{o}n-Zygmund estimates also give that
\begin{equation}\label{nablaPFest}
\|\nabla P_{F}\|_{L^{\frac{3}{2}}_{t}L^{\frac{11}{6}}_{x}(\mathbb{R}^3\times [-1, 0] )}\lesssim e^{-e^{\mathcal{M}}}.
\end{equation}
\textbf{Step 2: Finding a good scale with pigeonholing arguments}\\
Now we fix any $\mu\geq 1$ and $R_0$ such that
\begin{equation}\label{R0conditionT1=1}
1\leq R_{0}\leq  6e^{e^{2\mathcal{M}}}e^{-4\mu \mathcal{M}^{11\mu+22}}.
\end{equation}
From \eqref{VPVotimesVmainest}, we see that
$$\sum_{k=0}^{\infty}\int\limits_{-1}^{0}\int\limits_{\{x: \lambda^{k\mu}R_0<|x|<\lambda^{\mu(k+1)}R_0\}}|V(y,s)|^{\frac{10}{3}}+|P_{V\otimes V}(y,s)|^{\frac{5}{3}} dyds\leq \lambda:=\mathcal{M}^{11}. $$
By the pigeonhole principle, there exists 
\begin{equation}\label{k0set}
k_0\in\{0,1,\ldots \lfloor{\lambda^{\mu+1}\rfloor}\}
\end{equation}
such that
\begin{equation}\label{goodscale}
\int\limits_{-1}^{0}\int\limits_{\{x: \lambda^{k_0\mu}R_0<|x|<\lambda^{\mu(k_0+1)}R_0\}}|V(y,s)|^{\frac{10}{3}}+|P_{V\otimes V}(y,s)|^{\frac{5}{3}} dyds\leq \lambda^{-\mu}=\mathcal{M}^{-11\mu}.
\end{equation}
\textbf{Step 3: Concluding with $\varepsilon$-regularity}\\
Next, we define \begin{equation}\label{R1def}
R_{1}:=\lambda^{k_0\mu}R_0
\end{equation}
From now on we impose the restriction $\mathcal{M}\geq \max(\mathcal{M}_{0},\mathcal{M}(\mu),5^{\frac{1}{11\mu}})$ and we define
\begin{equation}\label{Adef}
A:=\{x:R_1+1<|x|<\lambda^{\mu}R_1-1\}.
\end{equation} 
This, \eqref{goodscale} and H\"{o}lder's inequality gives that
\begin{equation}\label{CKNminusPF}
\sup_{x_*\in A}\int\limits_{-1}^{0}\int\limits_{B(x_*,1)} |V(y,s)|^3+|P_{V\otimes V}(y,s)-(P_{V\otimes V})_{B(x_*,1)}(s)|^{\frac{3}{2}} dyds\lesssim \lambda^{-\frac{9\mu}{10}}\leq \lambda^{-\frac{9\mu}{11}} .
\end{equation}
Combining this with \eqref{nablaPFest} and Poincar\'{e}'s inequality gives that for $\mathcal{M}$ sufficiently large, there exists a positive universal constant $C_{univ}$ such that
\begin{equation}\label{CKNgoodscaleVP}
\sup_{x_*\in A}\int\limits_{-1}^{0}\int\limits_{B(x_*,1)} |V(y,s)|^3+|P(y,s)-(P)_{B(x_*,1)}(s)|^{\frac{3}{2}} dyds\lesssim \mathcal{M}^{-{9\mu}}+e^{-\tfrac{3}{2}e^{\mathcal{M}}}\leq C_{univ}\mathcal{M}^{-{9\mu}}.
\end{equation}
Thus for 
\begin{equation}\label{Mlargerlambdamu}
\mathcal{M}\geq \lambda_{0}(\mu):=\max\Big(\mathcal{M}_{0},\mathcal{M}(\mu),5^{\frac{1}{11\mu}},\Big(\frac{C_{univ}}{\varepsilon^{*}_1}\Big)^{\frac{1}{9\mu}}\Big)
\end{equation} we have 
\begin{equation}\label{VPCKNfinal}
\sup_{x_*\in A}\int\limits_{-1}^{0}\int\limits_{B(x_*,1)} |V(y,s)|^3+|P(y,s)-(P)_{B(x_*,1)}(s)|^{\frac{3}{2}} dyds\leq C_{univ}\mathcal{M}^{-{9\mu}}\leq \varepsilon^{*}_{1}.
\end{equation}
Here, $\mathcal{M}_{0}$ is a sufficiently large universal constant and $\varepsilon^{*}_{1}$ is as in Proposition \ref{CKNquanthigher}. Using \eqref{k0set} and \eqref{R1def}, we see
that
\begin{equation}\label{annulusbound}
 R_0 \leq R_1 \leq R_0\lambda^{\mu \lambda^{\mu+1}}\leq R_0e^{\mu \lambda^{\mu+2}}=R_0e^{\mu \mathcal{M}^{11\mu+22}}\quad\textrm{and}\quad \lambda^{\mu} R_1\leq R_0 e^{2\mu \lambda^{\mu+2}}=R_0 e^{2\mu \mathcal{M}^{11\mu+22}}.
\end{equation}
Using this and \eqref{R0conditionT1=1} gives that
\begin{equation}\label{setinclusiongoodannulus}
\{x: R_1<|x|<\lambda^{\mu}R_1\}\subset B(0, 6e^{e^{2\mathcal{M}}})\subseteq B(0,R).
\end{equation}
Here, $R$ is as in Proposition \ref{truncationprocedure}. Combining \eqref{setinclusiongoodannulus} and \eqref{forcecompactsupport}, we deduce that $(V,P)$ solves the Navier-Stokes equations on $\{x: R_1<|x|<\lambda^{\mu}R_1\}\times (-1,0).$ This and \eqref{VPCKNfinal} allows us to apply Proposition \ref{CKNquanthigher} to get that for $j=0,1$
\begin{equation}\label{VOmegaboundsA}
\|\nabla ^j V\|_{L^{\infty}(A\times (-\tfrac{1}{16},0))}\lesssim \mathcal{M}^{-3\mu}\quad\textrm{and}\quad \|\nabla \Omega\|_{L^{\infty}(A\times (-\tfrac{1}{16},0))}\lesssim \mathcal{M}^{-3\mu}.
\end{equation}
From \eqref{R0conditionT1=1}, \eqref{Mlargerlambdamu} and \eqref{annulusbound}, we see that for $R_{2}:=2R_1$
\begin{equation}\label{R2bounds}
R_1+1\leq R_2<\frac{\lambda^{\mu}R_2}{4}\leq \lambda^{\mu}R_1-1\quad\textrm{and}\quad  2R_0 \leq R_2 \leq 2R_0e^{\mu \mathcal{M}^{11\mu+22}}.
\end{equation}
We then see that \eqref{VOmegaboundsA} implies that
for $j=0,1$
\begin{equation}\label{VOmegaboundsAfinal}
\|\nabla ^j V\|_{L^{\infty}_{x,t}(R_2<|x|<\frac{R_2\mathcal{M}^{11\mu}}{4}\times (-\tfrac{1}{16},0))}\lesssim \mathcal{M}^{-3\mu}\quad\textrm{and}\quad \|\nabla \Omega\|_{L^{\infty}_{x,t}(R_2<|x|<\frac{R_2\mathcal{M}^{11\mu}}{4}\times (-\tfrac{1}{16},0))}\lesssim \mathcal{M}^{-3\mu}
\end{equation}
as required.
\end{proof}
\section{Proof of Theorem \ref{locest}}
We will show that the assumptions of Theorem \ref{locest} imply that
\begin{equation}\label{quantestreduce}
\sup_{t\in [-e^{-e^{e^{\mathcal{M}^{1126}}}},0]}|v(0,t)|\leq e^{e^{e^{\mathcal{M}^{1126}}}}.
\end{equation}
A translation and scaling argument (replacing $\mathcal{M}$ with $\mathcal{M}'=\frac{64}{49}\mathcal{M}$) then allows us to deduce the desired estimate \eqref{locesteqn} in Theorem \ref{locest}.

 From now on we will consider $(V,P)$, $F$ and $R$ as in Proposition \ref{truncationprocedure}, which are derived from $(v,p)$. Recall from Proposition \ref{truncationprocedure} that
 \begin{equation}\label{VPrecall}
 V(x,t)=\lambda v(\lambda x,\lambda^2 t)\quad\textrm{on}\quad B(0,R)\times [-2,0]\quad\textrm{and}\quad \|V\|_{L^{\infty}_{t}L^{3}_{x}(B(0,R)\times (-2,0))}\leq\mathcal{M}
 \end{equation}
 with
 \begin{equation}\label{Rlambda}
 R\geq 6e^{e^{2\mathcal{M}}}\quad\textrm{and}\quad \lambda=e^{-e^{2\mathcal{M}}}e^{-\mathcal{M}^5}.
 \end{equation}
 Recall that $(V,P)$ are smooth on $B(0,R)\times [-2,0]$, solve the Navier-Stokes equations on $B(0,R)\times [-2,0]$ and satisfy \eqref{VPmorrey}. Hence, we see that $(V,P)$ obey the hypothesis \eqref{vscaleinvariantvort} of Proposition \ref{vortlocrescale} with $M$ replaced by $\mathcal{M}^4$. Using \eqref{VPrecall}-\eqref{Rlambda} and Proposition \ref{vortlocrescale}, we see that in order to prove \eqref{quantestreduce} it suffices to show the following for $\Omega:=\nabla\times V$. Namely,
 \begin{equation}\label{OmegaNOconc}
 \int\limits_{B(0,\mathcal{M}^{96}(-t_{\mathcal{M}})^{\frac{1}{2}})} |\Omega(x,t_{\mathcal{M}}) |^2dx\leq \frac{\mathcal{M}^{-88}}{(-t_{\mathcal{M}})^{\frac{1}{2}}}\quad\textrm{with}\quad -t_{\mathcal{M}}:=\frac{\mathcal{M}^{-401}}{16}\exp\{-2\mathcal{M}^{1126}\exp\exp(\mathcal{M}^{1125})\}.
 \end{equation} 
 In order to show this, we assume that there exists
 \begin{equation}\label{-t'0cond}
 0<-t'_0<\frac{\mathcal{M}^{-401}}{8} 
 \end{equation}
  such that
 \begin{equation}\label{e.conct_0'bis}
\int\limits_{B(0,\mathcal{M}^{96}(-t_0')^\frac12)}|\Omega(x,t_0')|^2\, dx> \frac{\mathcal{M}^{-88}}{(-t_0')^{\frac12}}.
\end{equation}
We then show that this necessarily implies that
\begin{equation}\label{t01lowerbd}
-t_0'\geq \frac{\mathcal{M}^{-401}}{8}\exp\{-2\mathcal{M}^{1126}\exp\exp(\mathcal{M}^{1125})\},
\end{equation}
which then gives \eqref{OmegaNOconc}.

Assuming \eqref{e.conct_0'bis}, we see that Proposition \ref{backpropvortloc} then implies that
\begin{equation}\label{e.conct_0''bis}
\int\limits_{B(0,\mathcal{M}^{96}(-t'')^\frac12)}|\Omega(x,t'')|^2\, dx> \frac{\mathcal{M}^{-88}}{(-t'')^{\frac12}}.
\end{equation}
at any well-separated backward time
\begin{equation}\label{e.wellsepbis}
 t''\in [-\mathcal{M}^{-192},t_0']\quad\textrm{such}\,\,\textrm{that}\quad\mathcal{M}^{200}t_0'>t''.
\end{equation}
The rest of the proof relies on quantitative Carleman inequalities \cite[Proposition 4.2-4.3]{tao}\footnote{For reasons relating to notation we will refer to \cite[Proposition B.1-B.2]{barkerprange2020}, which are directly taken from Tao's paper \cite[Proposition 4.2-4.3]{tao} but adopt different notation.}. These are the tools used to transfer the concentration  information \eqref{e.conct_0''bis} from the time $t''$ 
to time $0$ and from the small scales $B(0,\mathcal{M}^{96}(-t'')^\frac12)$ to larger scales, which are smaller than $R$ in Proposition \ref{truncationprocedure}. 

The proof of Theorem {\ref{locest}} using quantitative Carleman inequalities closely follows \cite{barkerprange2020}, which is a physical space analogue of Tao's strategy \cite{tao}. In order to not to fully repeat parts of the proof in \cite{barkerprange2020}, we provide a detailed sketch. This allows the reader with \cite{barkerprange2020} in hand to recover all the details. Throughout the detailed sketch, $C$ denotes a positive universal constant, which may vary from line to line unless otherwise specified. 

\noindent{\bf Step 1: quantitative unique continuation.} The purpose of this step is to prove the following estimate:
\begin{align}\label{e.claimstep1}
T_1^{\frac12}e^{-\frac{C\mathcal{M}^{147}R^2_*}{T_1}}
\lesssim\ & \int\limits_{-T_1}^{-\frac{T_1}2}\int\limits_{B(0,2R_*)\setminus B(0,R_*/2)}|\Omega(x,t)|^2\, dxdt,
\end{align}
for all $T_1$ and $R_*$ such that 
\begin{equation}\label{e.wellsepT1}
2\mathcal{M}^{200}(-t_0')<T_1< \mathcal{M}^{-192}\quad\mbox{and}\quad R\geq e^{e^{\mathcal{M}}}\geq R_*\geq \mathcal{M}^{100}\Big(\frac{T_1}{2}\Big)^\frac12.
\end{equation} 
Denote $I_1:=(-T_1,-\frac{T_1}{2})$ and $T_1:=-t_0''$. We see that \eqref{e.wellsepT1} implies that \eqref{e.wellsepbis} is satisfied for every $t\in (-T_1,-\frac{T_1}{2})$. Hence, \eqref{e.conct_0''bis} is satisfied for every $t\in (-T_1,-\frac{T_1}{2})$.
Using that $(-T_1,-\frac{T_1}{2})\subset [-1,0]\subset [-2,0]$, we can apply Proposition \ref{epoch}.
This implies that there exists an epoch of regularity $I_1''=[t_1''-T_1'',t_1'']\subset I_1$ such that 
\begin{equation}\label{e.sizeI1'}
T_1''=|I_1''|={\mathcal{M}^{-46}}|I_1|={\mathcal{M}^{-46}}\tfrac{T_1}{2}
\end{equation}
and for $j=0, 1, 2$,
\begin{equation}\label{e.epochI1'}
\|\nabla^j V\|_{L^{\infty}_{x,t}(B(0,3e^{e^{2\mathcal{M}}})\times I_1'')}\lesssim \frac{1}{\mathcal{M}} |I_1''|^{\frac{-(j+1)}{2}}=\frac{1}{\mathcal{M}}(T_1'')^{\frac{-(j+1)}{2}}.
\end{equation}
Let $T_1''':=\frac34T_1''$ and $s''\in [t_1''-\frac{T_1''}4,t_1'']$. Let $x_1\in\R^3$ be such that $e^{e^{\mathcal{M}}}\geq |x_1|\geq \mathcal{M}^{100}(\frac{T_1}2)^\frac12$ and let 
$r_1:=\mathcal{M}^{50}|x_1|\geq \mathcal{M}^{150}(\frac{T_1}2)^\frac12$. 
Notice that for $\mathcal{M}$ large enough 
\begin{equation}\label{e.condr1}
r_1:=\mathcal{M}^{50}|x_1|\geq \mathcal{M}^{150}\Big(\frac{T_1}2\Big)^\frac12\geq \mathcal{M}^{53}(\mathcal{M}^{96}(-t_0'')^\frac12)\quad\textrm{and}\quad r_1^2\geq 4000T_1'''.
\end{equation}
Notice that
\begin{equation}\label{ballx1inclus}
B(x_1,r_1)\subset B(0,\mathcal{M}^{51}|x_1|)\subset B(0, 3e^{e^{2\mathcal{M}}}). 
\end{equation}
Using the above facts we can apply the Carleman inequality corresponding to quantitative unique continuation \cite[Proposition B.2]{barkerprange2020} on the cylinder $\mathcal C_1=\{(x,t)\in\mathbb R^3\times\mathbb R\, :\ t\in [0,T_1'''],\ |x|\leq r_1\}$, We apply this to the function $w:\, \mathcal C_1\rightarrow\R^3$, defined for all $(x,t)\in \mathcal C_1$ by 
\begin{equation*}
w(x,t):=\Omega(x_1+x,s''-t),
\end{equation*}
with
\begin{equation*}
r=r_1,\quad S=S_1:=T_1''',\quad C_{Carl}=\tfrac{16}{3},\quad \overline{s}_1=\frac{T_1'''}{20000}\quad\textrm{and}\quad \underline s_1=\mathcal{M}^{-150}T_1'''\leq \frac{T_1'''}{10000}.
\end{equation*}
 This gives 
\begin{equation}\label{e.estXYZ1}
Z_1\lesssim e^{-\frac{r_1^2}{500\overline s_1}}X_1+(\overline s_1)^\frac32\Big(\frac{e\overline s_1}{\underline s_1}\Big)^{\frac{Cr_1^2}{\overline s_1}}Y_1,
\end{equation}
where
\begin{align*}
&X_1:=\int\limits_{s''-T_1'''}^{s''}\int\limits_{B(x_1,\mathcal{M}^{50}|x_1|)}((T_1''')^{-1}|\Omega|^2+|\nabla\Omega|^2)\, dxds,\quad Y_1:=\int\limits_{B(x_1,\mathcal{M}^{50}|x_1|)}|\Omega(x,s'')|^2(\underline s_1)^{-\frac32}e^{-\frac{|x-x_1|^2}{4\underline s_1}}\, dx,
\\&Z_1:=\int\limits_{s''-\frac{T_1'''}{10000}}^{s''-\frac{T_1'''}{20000}}\int\limits_{B(x_1,\frac{\mathcal{M}^{50}|x_1|}2)}((T_1''')^{-1}|\Omega|^2+|\nabla\Omega|^2)e^{-\frac{|x-x_1|^2}{4(s''-s)}}\, dxds.
\end{align*}
 By \eqref{e.condr1}, we have
\begin{equation*}
B(0,\mathcal{M}^{96}(-s)^\frac12)\subset B(0,\mathcal{M}^{96}(-t''_0)^\frac12)\subset B(0,|x_1|)\subset  B(x_1,2|x_{1}|)\subset B\Big(x_1,\frac{\mathcal{M}^{50}|x_1|}{2}\Big)
\end{equation*}
for all $s\in [s''-\frac{T_1'''}{10000},s''-\frac{T_1'''}{20000}]$ and for $\mathcal{M}$ sufficiently large.
Combining this with \eqref{e.conct_0''bis} gives
\begin{align*}
Z_1\gtrsim\ &\int\limits_{s''-\frac{T_1'''}{10000}}^{s''-\frac{T_1'''}{20000}}\int\limits_{B(0,\mathcal{M}^{96}(-s)^\frac12)}(T_1''')^{-1}|\Omega(x,s)|^2\, dxds\, e^{-\frac{C|x_1|^2}{T_1'''}}\gtrsim 
 \mathcal{M}^{-111}(T_1'')^{-\frac{1}{2}}e^{-\frac{C|x_1|^2}{T_1''}}.
\end{align*}
Using \eqref{ballx1inclus}, we use the quantitative regularity \eqref{e.epochI1'} we obtain
\begin{align*}
e^{-\frac{r_1^2}{500\overline s_1}}X_1\lesssim\ & (T_1'')^{-2}\mathcal{M}^{150}|x_1|^3e^{-\frac{C\mathcal{M}^{100}|x_1|^2}{T_1''}}
\lesssim (T_1'')^{-\frac12}e^{-\frac{C\mathcal{M}^{100}|x_1|^2}{T_1''}}
\end{align*}
and
\begin{align*}
(\overline s_1)^\frac32\Big(\frac{e\overline s_1}{\underline s_1}\Big)^{\frac{Cr_1^2}{\overline s_1}}Y_1
\lesssim\ &\mathcal{M}^{225}e^{\frac{C\mathcal{M}^{101}|x_1|^2}{T_1''}}\int\limits_{B(x_1,\frac{|x_1|}2)}|\Omega(x,s'')|^2\, dx
+\mathcal{M}^{225}(T_1'')^{-\frac12}e^{-\frac{C\mathcal{M}^{150}|x_1|^2}{T_1''}}.
\end{align*}
Gathering these bounds and combining with \eqref{e.estXYZ1} yields
\begin{multline*}
\mathcal{M}^{-111}(T_1'')^{-\frac12}e^{-\frac{C|x_1|^2}{T_1''}}\lesssim (T_1'')^{-\frac12}e^{-\frac{C\mathcal{M}^{100}|x_1|^2}{T_1''}}\\
+\mathcal{M}^{225}e^{\frac{C\mathcal{M}^{101}|x_1|^2}{T_1''}}\int\limits_{B(x_1,\frac{|x_1|}2)}|\Omega(x,s'')|^2\, dx+\mathcal{M}^{225}(T_1'')^{-\frac12}e^{-\frac{C\mathcal{M}^{150}|x_1|^2}{T_1''}}.
\end{multline*}
Using \eqref{e.sizeI1'} and $e^{e^{\mathcal{M}}}\geq|x_1|\geq \mathcal{M}^{100}(\frac{T_1}{2})^{\frac{1}{2}}$, we see that for $\mathcal{M}$ sufficiently large
\begin{align*}
\mathcal{M}^{-336}(T_1'')^{-\frac12}e^{-\frac{C\mathcal{M}^{101}|x_1|^2}{T_1''}}\lesssim\ & \int\limits_{B(x_1,\frac{|x_1|}2)}|\Omega(x,s'')|^2\, dx.
\end{align*}
Hence, 
for all $s''\in [t_1''-\frac{T_1''}{4},t_1'']$, for all $e^{e^{\mathcal{M}}}\geq |x_1|\geq \mathcal{M}^{100}(\frac{T_1}{2})^\frac12$,
\begin{equation*}
\int\limits_{B({x_1},\frac{|x_1|}{2})}|\Omega(x,s'')|^2\, dx\gtrsim \mathcal{M}^{-336}(T_1'')^{-\frac12}e^{-\frac{C\mathcal{M}^{101}|x_1|^2}{T_1''}}=\mathcal{M}^{-336}(T_1'')^{-\frac12}e^{-\frac{2C\mathcal{M}^{147}|x_1|^2}{T_1}}.
\end{equation*}
Let $e^{e^{\mathcal{M}}}\geq R_*\geq \mathcal{M}^{100}(\frac{T_1}2)^\frac12$ and $x_{1}\in\mathbb{R}^3$ be such that $|x_{1}|=R_*$. 
Integrating in time $[t_1''-\frac{T_1''}4,t_1'']$ yields 
the claim \eqref{e.claimstep1} of Step 1.

\noindent{\bf Step 2: quantitative backward uniqueness.} The goal of this step and Step 3 below is to prove the following claim:
\begin{align}\label{e.claimstep2}
\begin{split}
T_2^{-\frac12}\exp\big(-\exp({\mathcal{M}^{1123}})\big)\lesssim\ &\int\limits_{B\big(0,\tfrac{3}{16}\mathcal{M}^{1100}R_{2}\big)\setminus B(0,2R_{2})}|\Omega(x,0)|^2\, dx,
\end{split}
\end{align}
for all ${8}\mathcal{M}^{401}(-t_0')< T_2\leq 1$ and $\mathcal{M}$ sufficiently large. Here, $R_{0}$ and $R_{2}$  are as in \eqref{e.restrR2}-\eqref{e.defann2}. This is the key estimate for Step 4 below and the proof of Theorem \ref{locest}.

 Let $T_1$ and $T_2$ be such that 
\begin{equation}\label{e.sepT2}
{8}\mathcal{M}^{401}(-t_0')< T_2\leq 1\quad\mbox{and}\quad T_1:=\frac{T_2}{4\mathcal{M}^{201}}.
\end{equation}
Let 
\begin{equation}\label{e.restrR2}
R_0:=K^\sharp(T_2)^\frac12,
\end{equation}
for a universal constant $K^\sharp\geq 1$ to be chosen sufficiently large below. In particular it is chosen in Step 3 such that \eqref{e.choiceKsharp} holds, which makes it possible to absorb the upper bound \eqref{e.controlX3} of $X_3$ in the left hand side of \eqref{e.conclcarltwobiswithuniversal}. From \eqref{e.restrR2}, we see that for $\mathcal{M}$ sufficiently large we have that
$$T_2^{\frac 12}\leq R_0\leq T_2^{\frac{1}{2}}6e^{e^{2\mathcal{M}}}e^{-400\mathcal{M}^{1122}} .$$ 
By Proposition \ref{annulus}, for $\mathcal{M}\geq \lambda_{0}(100)$ there exists a scale
\begin{equation}\label{e.bdR2prime}
2R_0\leq R_2\leq 2R_0e^{{100 \mathcal{M}^{1122}}}
\end{equation}
and a good cylindrical annulus 
\begin{equation}\label{e.defann2}
\mathcal A_2:=\{R_2<|x|<\tfrac{{\mathcal{M}^{1100}}R_2}{4}\}\times\Big(-\frac{T_2}{16},0\Big)
\end{equation} 
such that for $j=0,1$, 
\begin{align}\label{e.linftyderivstep2}
\begin{split}
\|\nabla^j V\|_{L^{\infty}(\mathcal A_2)}
\leq\  \bar{C}_{j} \mathcal{M}^{-{300}
}T_2^{-\frac{j+1}{2}},\quad
\|\nabla \Omega\|_{L^{\infty}(\mathcal A_2)}
\leq\ \bar{C}_{2} \mathcal{M}^{-{300}
}T_2^{-\frac32}.
\end{split}
\end{align}
Let
\begin{equation}\label{e.defann2bis}
\widetilde{\mathcal A}_2:=\Big\{4R_2\leq|x|\leq\frac{{\mathcal{M}^{1100}}R_2}{16}\Big\}\times\Big[-\frac{T_2}{\mathcal{M}^{201}},0\Big]\subseteq B(0,R)\times \Big[-\frac{T_2}{\mathcal{M}^{201}},0\Big]
\end{equation}
and choose $\mathcal{M}$ sufficiently large such that $\bar{C}_{j} \mathcal{M}^{-{300}}\leq 1$ and $\bar{C}_{2} \mathcal{M}^{-{300}}\leq 1$. Using \eqref{e.defann2bis} we see that $(V,P)$ solve the Navier-Stokes equations on $\widetilde{\mathcal A}_2$.\\ 
The above facts allow us to apply the Carleman inequality corresponding to quantitative backward uniqueness \cite[Proposition B.1]{barkerprange2020} 
to the function $w:\, \widetilde{\mathcal A}_2\rightarrow\R^3$ defined for all $(x,t)\in\widetilde{\mathcal A}_2$ by
\begin{equation*}
w(x,t)=\Omega(x,-t),
\end{equation*}
with
\begin{equation*}
r_{-}:=4R_2,\quad r_{+}:=\tfrac{1}{16}\mathcal{M}^{1100}R_2,\quad S=S_2:=\frac{T_2}{\mathcal{M}^{201}}\quad\textrm{and}\quad C_{Carl}=\mathcal{M}^{201}.
\end{equation*}
This gives

\begin{equation}\label{e.conclcarlonebis}
Z_2\lesssim e^{-\frac{C\mathcal{M}^{1100}(R_2)^2}{T_2}}\big(X_2+e^{\frac{C\mathcal{M}^{2200}(R_2)^2}{T_2}}Y_2\big),
\end{equation}
where 
\begin{align*}
X_2:=\ &\int\limits_{-\frac{T_{2}}{\mathcal{M}^{201}}}^{0}\int\limits_{r_{-}\leq |x|\leq r_{+}}e^{\frac{4|x|^2}{T_2}}(\mathcal{M}^{201}T_2^{-1}|\Omega|^2+|\nabla \Omega|^2)\, dxdt,\quad 
Y_2:=\ &\int\limits_{r_-\leq |x|\leq r_+}|\Omega(x,0)|^2\, dx,\\
Z_2:=\ &\int\limits_{-\frac{T_2}{4 \mathcal{M}^{201}}}^{0}\int\limits_{10r_-\leq|x|\leq \frac{r_+}2}(\mathcal{M}^{201}T_2^{-1}|\Omega|^2+|\nabla \Omega|^2)\, dxdt.
\end{align*}
Notice that for $\mathcal{M}$ sufficiently large we have that $$B(0,40r_{-})\setminus B(0,10r_{-})\subset \{10r_{-}\leq |x|\leq \tfrac{r_{+}}{2}\}.$$
Next we use this, the separation condition \eqref{e.sepT2} and the fact that for $\mathcal{M}$ large enough \eqref{e.restrR2} implies
\begin{equation*}
e^{e^{\mathcal{M}}}\geq 20r_{-}\geq 10R_2\geq 20R_0=20K^\sharp T_2^\frac12\geq \mathcal{M}^{100}\Big(\frac{T_2}{8\mathcal{M}^{201}}\Big)^\frac12=\mathcal{M}^{100}\Big(\frac{T_1}{2}\Big)^\frac12.
\end{equation*} 
This gives that we can apply the concentration result of Step 1, taking there $T_1=\frac{T_2}{4 \mathcal{M}^{201}}=\frac {S_2}4$ and $R_{*}=20r_-$. By \eqref{e.claimstep1} we have that
\begin{equation}\label{e.lowerZ2}
Z_2\gtrsim \mathcal{M}^{201}\left(\frac{T_2}{4\mathcal{M}^{201}}\right)^\frac12e^{-\frac{C\mathcal{M}^{348}(R_2)^2}{T_2}}T_2^{-1}\gtrsim T_2^{-\frac12}e^{-\frac{C\mathcal{M}^{348}(R_2)^2}{T_2}}.
\end{equation}
Therefore, one of the following two lower bounds holds
\begin{align}
T_2^{-\frac12}\exp\Big(\frac{C\mathcal{M}^{1100}(R_2)^2}{T_2}\Big)\lesssim X_2,\label{e.step2lowerbd1}\\
T_2^{-\frac12}\exp(-\exp(\mathcal{M}^{1123}))
\lesssim e^{-\frac{C\mathcal{M}^{2200}(R_2)^2}{T_2}}T_{2}^{-\frac12}\lesssim Y_2,\label{e.step2lowerbd2}
\end{align}
where we used the upper bound \eqref{e.bdR2prime} for \eqref{e.step2lowerbd2}. 
The bound \eqref{e.step2lowerbd2} can be used directly in Step 4 below. On the contrary, if \eqref{e.step2lowerbd1} holds more work needs to be done to transfer the lower bound on the enstrophy at time $0$. This is the objective of Step 3 below.

\noindent{\bf Step 3: a final application of quantitative unique continuation.} Assume that the bound \eqref{e.step2lowerbd1} holds. The same reasoning as in \cite{tao} and the subsequent paper \cite{barkerprange2020} (involving the pigeonhole principle three times) gives the following. There exists 
\begin{equation}\label{e.condR3}
8R_2\leq R_3\leq \tfrac18{\mathcal{M}^{1100}}R_2,
\end{equation}
\begin{equation}\label{e.condt3}
\frac{1}{2}\exp\Big(-\frac{8R_3^2}{T_2}\Big)T_2\leq -t_3\leq\frac{T_2}{\mathcal{M}^{201}}
\end{equation}
and $x_3\in B(0,R_3)\setminus B(0,\frac{R_3}2)$ such that
 
\begin{equation}\label{e.conct3}
T_2^{-\frac12}\exp\Big(-\frac{18R_3^2}{T_2}\Big)\lesssim \int\limits_{2t_3}^{t_3}\int\limits_{B(x_3,(-t_3)^\frac12)}(T_2^{-1}|\Omega|^2+|\nabla \Omega|^2)\, dxdt.
\end{equation}
We apply now the Carleman inequality corresponding to quantitative unique continuation \cite[Proposition B.2]{barkerprange2020}. We apply this to the function $w:\, \overline{B(0,r)}\times [0,-20000t_{3}]\rightarrow\R^3$ defined for all $(x,t)\in\overline{B(0,r)}\times [0,-20000t_{3}]$ by
\begin{equation*}
w(x,t)=\Omega(x+x_3,-t),
\end{equation*}
with 
\begin{equation}\label{e.choicer_3}
S=S_3:=-20000t_{3},\quad r=r_3:=1000R_{3}\Big(-\frac{t_{3}}{T_{2}}\Big)^{\frac{1}{2}},\quad \overline s_3=\underline s_3=-t_3\quad\textrm{and}\quad C_{Carl}=1.
\end{equation}
This gives
\begin{equation}\label{e.conclcarltwobis}
Z_3\lesssim e^{\frac{r_3^2}{500t_3}}X_3+ (-t_3)^\frac32e^{-\frac{Cr_3^2}{t_3}}Y_3,
\end{equation}
where 
\begin{align*}
X_3:=\ &\int\limits_{-S_{3}}^0\int\limits_{B(x_3,r_3)}(S_{3}^{-1}|\Omega|^2+|\nabla\Omega|^2)\, dxdt,\qquad Y_3:=\int\limits_{B(x_3,r_3)}|\Omega(x,0)|^2(-t_3)^{-\frac32}e^{\frac{|x-x_3|^2}{4t_3}}\, dx,\\
Z_3:=\ &\int\limits_{2t_3}^{t_3}\int\limits_{B(x_3,\frac{r_3}2)}(S_{3}^{-1}|\Omega|^2+|\nabla \Omega|^2)e^{\frac{|x-x_3|^2}{4t}}\, dxdt.
\end{align*}
Using \eqref{e.restrR2}-\eqref{e.bdR2prime} and \eqref{e.condR3}-\eqref{e.condt3}, we have that \begin{multline}\label{domaininclus}
B(x_3,(-t_3)^\frac12)\subset B(x_3,\tfrac{r_3}2)\subset B(x_3,r_{3})\subset B\Big(x_3,\frac{|x_{3}|}{2}\Big)\\
\subset\{\tfrac{R_3}{4}<|y|<\tfrac{3}{2}R_3\}\subset\{2R_2<|y|<\tfrac{3}{16}\mathcal{M}^{1100}R_2\}. 
\end{multline}
Using this and \eqref{e.conct3} and $T_{2}^{-1}\leq S_{3}^{-1}$ we have
\begin{equation}\label{Z3lowerbound}
  T_2^{-\frac12}\exp\Big(-\frac{18R_3^2}{T_2}\Big)\lesssim \int\limits_{2t_3}^{t_3}\int\limits_{B(x_3,(-t_{3})^{\frac{1}{2}})}(T_2^{-1}|\Omega|^2+|\nabla \Omega|^2)e^{\frac{|x-x_3|^2}{4t}}\, dxdt\leq Z_{3}
\end{equation} 
Hence, we have
\begin{equation}\label{e.conclcarltwobiswithuniversal}
T_2^{-\frac12}\exp\Big(-\frac{18R_3^2}{T_2}\Big)\leq C_{univ} e^{\frac{r_3^2}{500t_3}}X_3+ C_{univ}(-t_3)^\frac32e^{-\frac{Cr_3^2}{t_3}}Y_3.
\end{equation}
Where $C_{univ}$ is a positive universal constant.
Using the bounds \eqref{e.linftyderivstep2} along with \eqref{e.condt3}, we find that
\begin{multline}\label{e.controlX3}
C_{univ}e^{\frac{r_3^2}{500t_3}}X_3\lesssim S_{3}^{-2}r_{3}^3e^{\frac{r_3^2}{500t_3}}\lesssim {(-t_{3})}^{-\frac12}e^{\frac{r_3^2}{1000t_3}}\lesssim T_2^{-\frac12}e^{\frac{4R_3^2}{T_2}}e^{\frac{r_3^2}{1000t_3}}\\
\lesssim T_2^{-\frac12}e^{-\frac{996 R_3^2}{T_2}}\lesssim T_2^{-\frac12}e^{-\frac{18 R_3^2}{T_2}}e^{-\frac{978 R_3^2}{T_2}}\leq C'_{univ} T_2^{-\frac12}e^{-\frac{18 R_3^2}{T_2}}e^{-978\cdot 256(K^\sharp)^2}.
\end{multline}
We choose $K^\sharp$ sufficiently large so that 
\begin{equation}\label{e.choiceKsharp}
C'_{univ}e^{-978\cdot 256(K^\sharp)^2}\leq \frac12,
\end{equation}
where $C'_{univ}\in(0,\infty)$ is the universal constant appearing in the last inequality of \eqref{e.controlX3}. 
Therefore, the term in the right hand side of \eqref{e.controlX3} can be absorbed into the left hand side of \eqref{e.conclcarltwobiswithuniversal}. 
From this and \eqref{e.conclcarltwobiswithuniversal}, we obtain
\begin{equation*}
T_2^{-\frac12}\exp\Big(-C\frac{R_3^2}{T_2}\Big)\lesssim \int\limits_{B(x_3,r_3)}|\Omega(x,0)|^2\, dx.
\end{equation*}
Using \eqref{e.restrR2}, the upper bound 
\begin{equation*}
R_3\leq\tfrac{1}8\mathcal{M}^{1100}R_2\leq \tfrac{1}{4}\mathcal{M}^{1100}\exp(100\mathcal{M}^{1122})R_0
\end{equation*}
and \eqref{domaininclus}, it follows that 
\begin{align}\label{e.lastbddstep3}
T_2^{-\frac12}\exp(-\exp(\mathcal{M}^{1123}))
\lesssim\ &\int\limits_{B\big(0,\tfrac{3}{16}\mathcal{M}^{1100}R_{2}\big)\setminus B(0,2R_{2})}|\Omega(x,0)|^2\, dx.
\end{align}

\noindent{\bf Step 4, conclusion: summing the scales and lower bound for the global $L^3$ norm.} Next we use \eqref{e.lastbddstep3} and argue in a similar way to \cite{barkerprange2020}\footnote{These arguments are in turn based on Tao's arguments in \cite{tao}.}. In particular, the pigeonhole principle, the quantitative estimate \eqref{e.linftyderivstep2}, integration by parts and H\"{o}lder's inequality gives that

\begin{equation}\label{e.lowerL3}
\int\limits_{B\big(0,\exp(\mathcal{M}^{1123})T_2^\frac12\big)\setminus B\big(0,T_2^\frac12\big)}|V(x,0)|^3\, dx\geq \exp\big(-\exp(\mathcal{M}^{1125})\big),
\end{equation}
for all ${8}\mathcal{M}^{401}(-t_0')\leq T_2\leq 1$. We note that this holds for $T_{2}={8}\mathcal{M}^{401}(-t_0')$ using that $V$ is smooth. 

We then sum \eqref{e.lowerL3} over $k$ scales starting with $T_{2}={8}\mathcal{M}^{401}(-t_0')$, where
\begin{equation}\label{enoughscales1}
k:=\lfloor\tfrac12 \mathcal{M}^{-1123}\log\Big(\frac{\mathcal{M}^{-401}}{8(-t'_{0})}\Big)+1\rfloor.
\end{equation}
 This gives that
\begin{align*}
&\exp\big(-\exp({\mathcal{M}^{1125}})\big)\tfrac12 \mathcal{M}^{-1123}\log\Big(\frac{\mathcal{M}^{-401}}{8(-t'_{0})}\Big)
\leq &\int\limits_{B(0,\exp(\mathcal{M}^{1123}))\setminus B(0,({8}\mathcal{M}^{401}(-t_0'))^\frac12)}|V(x,0)|^3\, dx\leq \mathcal{M}^3.
\end{align*}
 Rearranging this gives the lower bound \eqref{t01lowerbd}, as required.

\section{Proof of Theorem \ref{locaquantrate}}
Without loss of generality, assume $0<\delta<\min(\dist(x_0,\partial B(4)),({\tfrac{1}{2}T_*})^{\frac{1}{2}})$. Assume the hypothesis of Theorem \ref{locaquantrate} holds with the contrapositive of \eqref{locquantrateeqn}. Then there exists $ N\in (0,\infty)$ such that
\begin{equation}\label{contraquantrate}
\sup_{0<s<t}\|v(\cdot,s)\|_{L^{3}(B(x_0,\delta)))}\leq N\Big(\log\log\log\Big(\frac{1}{(T_*-t)^{\frac{1}{4}}}\Big)\Big)^{\frac{1}{1129}}\quad\forall t\in (0,T_*). 
\end{equation}
Next, we fix $\varepsilon_0>0$ such that
\begin{equation}\label{epcond1}
\frac{N}{\Big(\log\log\log\Big({(\varepsilon_0)^{-\frac{1}{4}}}\Big)\Big)^{\frac{1}{1128}-\frac{1}{1129}}}\leq 1\quad\textrm{and}
\end{equation}
\begin{equation}\label{epcond2}
\Big(\log\log\log\Big({(\varepsilon_0)^{-\frac{1}{4}}}\Big)\Big)^{\frac{1}{1128}}\geq \max\{\mathcal{M}_{0}, {16}{\delta^{-2}}\|p\|_{L^{\frac{3}{2}}_{x,t}(B(0,4)\times (0,T_*))}^{\frac{3}{2}})\}.
\end{equation} 
Here, $\mathcal{M}_{0}$ is as in Theorem \ref{locest}. Then \eqref{contraquantrate}-\eqref{epcond1} gives that
\begin{equation}\label{contraconsequence}
\sup_{0<s<t}\|v(\cdot,s)\|_{L^{3}(B(x_0,\delta)))}\leq \Big(\log\log\log\Big(\frac{1}{(T_*-t)^{\frac{1}{4}}}\Big)\Big)^{\frac{1}{1128}}\quad\forall t\in (T_*-\varepsilon_0, T_*).
\end{equation}
From now on, we take any $t\in (\max(\tfrac{1}{2}T_*, T_*-\varepsilon_0), T_*)$.
Next we rescale
$$(V_{x_0,\delta}(y,s), P_{x_0,\delta}(y,s))=(\tfrac{\delta}{4}v(\tfrac{\delta}{4}x+x_0, \tfrac{\delta^2}{16}s+t),\tfrac{\delta^2}{16}p(\tfrac{\delta}{4}x+x_0, \tfrac{\delta^2}{16}s+t)). $$
Then, $(V_{x_0,\delta}, P_{x_0,\delta})$ is a smooth solution to the Navier-Stokes equations in $B(0,4)\times [-16,0]$. By \eqref{contraconsequence} we have
\begin{equation}\label{rescaleL3locbound}
\|V_{x_0,\delta}\|_{L^{\infty}_{t}L^{3}_{x}(B(0,4)\times (-16,0))}\leq \mathcal{M}:=\Big(\log\log\log\Big(\frac{1}{(T_*-t)^{\frac{1}{4}}}\Big)\Big)^{\frac{1}{1128}}.
\end{equation}
By \eqref{epcond2}, we have that
\begin{equation}\label{Mcond}
\max\{\mathcal{M}_{0}, \|P_{x_0,\delta}\|_{L^{\frac{3}{2}}_{x,t}(Q(0,4))}^{\frac{3}{2}}\}\leq \mathcal{M}.
\end{equation}
From \eqref{rescaleL3locbound}-\eqref{Mcond}, we see that we can apply Theorem \ref{locest}. This gives
$$\tfrac{\delta}{4}\|v(\cdot,t)\|_{L^{\infty}(B(x_0,\tfrac{\delta}{8}))}=\|V_{x_0,\delta}(\cdot,0)\|_{L^{\infty}(B(0,\frac{1}{2}))} \leq \exp(\exp(\exp(\mathcal{M}^{1128}))=\frac{1}{(T_*-t)^{\frac{1}{4}}}. $$
So \eqref{contraquantrate} implies that
\begin{equation}\label{vlocestsing} 
\|v(\cdot,t)\|_{L^{\infty}(B(x_0,\tfrac{\delta}{8}))}\leq \frac{4}{\delta(T_*-t)^{\frac{1}{4}}}\quad\forall t\in (\max(\tfrac{1}{2}T_*, T_*-\varepsilon_0), T_*).
\end{equation}
This implies that
\begin{equation}\label{vLPS}
v\in L^{2}_{t}L^{\infty}_{x}(B(x_0,\tfrac{\delta}{8})\times (\max(\tfrac{1}{2}T_*, T_*-\varepsilon_0), T_*)).
\end{equation} 
By \cite[Theorem 7]{greogryvladimirhandbook}, this implies that 
$$v\in L^{\infty}_{x,t}(B(x_0,r)\times (T_*-r^2,T_*)) $$
for all sufficiently small $r>0$. This violates the hypothesis of Theorem \ref{locaquantrate}. Hence the contrapositive \eqref{contraquantrate} cannot hold and we get the desired conclusion.\,\,$\Box$\\
\textbf{Acknowledgement}.\\
The author thanks Dallas Albritton and Christophe Prange for helpful discussions.
The author also wishes to thank Armen Shirikyan, Jin Tan and Nikolay Tzvetkov, whose questions during a talk of the author prompted me to write down these results.

\bibliography{refs}        
\bibliographystyle{plain}

\end{document}